\documentclass[a4paper]{amsart}
\usepackage{hyperref}
\usepackage{txfonts, amsmath,amstext,amsthm,amscd,amsopn,verbatim,amssymb, amsfonts, amsrefs}
\usepackage{fullpage}

\usepackage[bbgreekl]{mathbbol}
\usepackage{tikz}
\usetikzlibrary{matrix}
\usetikzlibrary{shapes}
\usetikzlibrary{arrows}
\usetikzlibrary{calc,3d}
\usetikzlibrary{decorations,decorations.pathmorphing, decorations.pathreplacing}
\usetikzlibrary{through}
\tikzset{ext/.style={circle, draw,inner sep=1pt},int/.style={circle,draw,fill,inner sep=1pt},nil/.style={inner sep=1pt}}
\tikzset{exte/.style={circle, draw,inner sep=3pt},inte/.style={circle,draw,fill,inner sep=3pt}}
\tikzset{diagram/.style={matrix of math nodes, row sep=3em, column sep=2.5em, text height=1.5ex, text depth=0.25ex}}
\tikzset{diagram2/.style={matrix of math nodes, row sep=0.5em, column sep=0.5em, text height=1.5ex, text depth=0.25ex}}
\tikzset{dedge/.style={draw,-latex}}

\usepackage{tikz-cd}

\theoremstyle{plain}
  \newtheorem{thm}{Theorem}[section]

  \newtheorem{prop}[thm]{Proposition}
  
  \newtheorem{cor}[thm]{Corollary}
  \newtheorem{lemma}[thm]{Lemma}
   
\theoremstyle{definition}
  \newtheorem{ex}[thm]{Example}
  \newtheorem{rem}[thm]{Remark}

\newcommand{\alg}[1]{\mathfrak{{#1}}}

\newcommand{\Hom}{\mathop{Hom}}

\newcommand{\R}{{\mathbb{R}}}

\newcommand{\K}{{\mathbb{K}}}
\newcommand{\Q}{{\mathbb{Q}}}

\newcommand{\HGC}{{\mathrm{HGC}}}

\newcommand{\Graphs}{{\mathsf{Graphs}}}

\newcommand{\fHGC}{{\mathrm{fHGC}}}

\newcommand{\hoPoiss}{\mathsf{hoPois}}

\newcommand{\Gra}{{\mathsf{Gra}}}

\newcommand{\dGra}{{\mathsf{dGra}}}
\newcommand{\dGraor}{\mathsf{dGra}^{\rm or}}
\newcommand{\Graor}{\dGraor}

\newcommand{\vGraor}{\mathsf{vGra}^{or}}
\newcommand{\vGra}{\mathsf{vGra}}

\newcommand{\Tw}{\mathit{Tw}}
\newcommand{\Def}{\mathrm{Def}}

\newcommand{\Poiss}{\mathsf{Pois}}
\newcommand{\EBr}{\mathsf{EBr}}

\newcommand{\op}{\mathcal}

\newcommand{\Br}{\mathsf{Br}}

\newcommand{\Lie}{\mathsf{Lie}}
\newcommand{\ELie}{\mathsf{ELie}}
\newcommand{\hoLie}{\mathsf{hoLie}}

\newcommand{\Ass}{\mathsf{Assoc}}
\newcommand{\Com}{\mathsf{Com}}

\newcommand{\FM}{\mathsf{FM}}

\newcommand{\bpm}{\begin{pmatrix}}
\newcommand{\epm}{\end{pmatrix}}
\newcommand{\Tpoly}{T_{\rm poly}}
\newcommand{\Dpoly}{D_{\rm poly}}

\newcommand{\GC}{\mathrm{GC}}

\newcommand{\dGC}{\mathrm{dGC}}
\newcommand{\fGC}{\mathrm{fGC}}

\newcommand{\dfGC}{\mathrm{dfGC}}
\newcommand{\fdGC}{\dfGC}
\newcommand{\GCor}{\mathrm{GC}^{or}}

\renewcommand{\Hom}{\mathrm{Hom}}
\newcommand{\MC}{\mathsf{MC}}

\newcommand{\KGra}{\mathsf{KGra}}
\newcommand{\KGraphs}{\mathsf{KGraphs}}
\newcommand{\fGCor}{\fGC^{\rm or}}
\newcommand{\fvGCor}{\mathsf{fvGC}^{\rm or}}
\newcommand{\vGCor}{\mathsf{vGC}^{\rm or}}
\newcommand{\vGC}{\mathsf{vGC}}
\newcommand{\fvGC}{\mathsf{fvGC}}

\newcommand{\KGraor}{\KGra^{\rm or}}
\newcommand{\KGraphsor}{\KGraphs^{\rm or}}
\newcommand{\Graphsor}{\Graphs^{\rm or}}
\newcommand{\dGraphsor}{\Graphsor}

\newcommand{\mU}{\mathcal{U}}

\DeclareMathOperator{\End}{End}

\newcommand{\hoe}{\mathsf{hoe}}
\newcommand{\e}{\mathsf{e}}

\newcommand{\La}{\Lambda}

\newcommand{\Embbar}{\overline{\mathrm{Emb}}}

\newcommand{\hoCom}{\mathrm{hoCom}}

\newcommand{\dgra}{\mathrm{dgra}}
\newcommand{\kgra}{\mathrm{kgra}}

\newcommand{\Grap}{\mathsf{Grap}}
\newcommand{\Pois}{\Poiss}
\newcommand{\EGer}{\mathsf{EGer}}

\begin{document}
\title{Deformation quantization and the Gerstenhaber structure on the homology of knot spaces }

 \author{Thomas Willwacher}
\address{Institute of Mathematics \\ University of Zurich \\  
Winterthurerstrasse 190 \\
8057 Zurich, Switzerland}
\email{thomas.willwacher@math.uzh.ch}

\thanks{The author has been partially supported by the Swiss National Science foundation, grant 200021\_150012, and the SwissMAP NCCR funded by the Swiss National Science foundation.}


\begin{abstract}
It is well known that (in suitable codimension) the spaces of long knots in $\R^n$ modulo immersions are double loop spaces. Hence the homology carries a natural Gerstenhaber structure, given by the Gerstenhaber structure on the Hochschild homology of the $n$-Poisson operad. In this paper, we compute the latter Gerstenhaber structure in terms of hairy graphs, and show that it is not quite trivial combinatorially. Curiously, the construction makes essential use of methods and results of deformation quantization, and thus provides a bridge between two previously little related subjects in mathematics.
\end{abstract}

\maketitle


\section{Introduction}
We denote by $\Poiss_n$ the $n$-Poisson operad. An algebra over $\Poiss_n$ is a vector space $V$ with a commutative product of degree 0 and a compatible Lie bracket of degree $1-n$.
In particular, $\Poiss_n$ is endowed with a map from the associative operad
\[
 \Ass\to \Com \to \Poiss_n
\]
and is hence a multiplicative operad. Thus it makes sense to consider the Hochschild complex 
\[
 CH(\Poiss_n)
\]
and the Hochschild homology 
\[
 HH(\Poiss_n) = H(CH(\Poiss_n)).
\]
By the Deligne conjecture (now a Theorem), the complex $CH(\Poiss_n)$ is naturally endowed with a homotopy Gerstenhaber (i.e., homotopy $\Poiss_2$-)algebra structure, and hence $HH(\Poiss_n)$ is a Gerstenhaber algebra. The product is the cup product, and the Lie bracket the Gerstenhaber bracket.

The objects $HH(\Poiss_n)$ and $CH(\Poiss_n)$ and their Gerstenhaber structure has recently received some attention by algebraic topologists because of the following Theorem.

\begin{thm}[Songhafouo Tsopm\'en\'e \cite{ST}, Moriya \cite{Moriya}]
For $n\geq 4$ there is an isomorphism of Gerstenhaber algebras between the homology of the space of long knots modulo immersions and the Hochschild homology of $\Poiss_n$,
\[
 H_*(\Embbar(\R,\R^n), \Q) \cong HH(\Poiss_n).
\]
\end{thm}

Currently, we do not have a complete understanding of the Hochschild homology $HH(\Poiss_n)$, even as a graded vector space. However, some amount of understanding can be obtained through the hairy graph complexes $\fHGC_{1,n}$. These graph complexes consists of (possibly infinite) $\Q$-linear combinations of graphs with one or more ``external legs'' or ``hairs'' like the following.
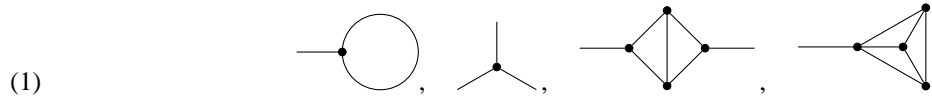
\begin{equation}\label{equ:HGCsample}
 \begin{tikzpicture}[scale=.5]
 \draw (0,0) circle (1);
 \draw (-180:1) node[int]{} -- +(-1.2,0);
 \end{tikzpicture}
,\quad
\begin{tikzpicture}[scale=.6]
\node[int] (v) at (0,0){};
\draw (v) -- +(90:1) (v) -- ++(210:1) (v) -- ++(-30:1);
\end{tikzpicture}
\,,\quad
\begin{tikzpicture}[scale=.5]
\node[int] (v1) at (-1,0){};\node[int] (v2) at (0,1){};\node[int] (v3) at (1,0){};\node[int] (v4) at (0,-1){};
\draw (v1)  edge (v2) edge (v4) -- +(-1.3,0) (v2) edge (v4) (v3) edge (v2) edge (v4) -- +(1.3,0);
\end{tikzpicture}
 \, ,\quad
 \begin{tikzpicture}[scale=.6]
\node[int] (v1) at (0,0){};\node[int] (v2) at (180:1){};\node[int] (v3) at (60:1){};\node[int] (v4) at (-60:1){};
\draw (v1) edge (v2) edge (v3) edge (v4) (v2)edge (v3) edge (v4)  -- +(180:1.3) (v3)edge (v4);
\end{tikzpicture}
 \, 
 \end{equation}
A more detailed definition will be given in section \ref{sec:hairydef} below.
The relation to $HH(\Poiss_n)$ is then as follows.
\begin{thm}[ Turchin \cite{Turchin1}]\label{thm:PoisHKR}
There is a zig-zag of quasi-isomorphism of complexes
\[
\fHGC_{1,n} \to \cdot \leftarrow CH(\Poiss_n).
\]
\end{thm}

The purpose of the present paper is to upgrade the above result to a quasi-isomorphism of homotopy Gerstenhaber algebras, for a homotopy Gerstenhaber algebra structure on $\fHGC_{1,n}$ which we define. In particular, we provide a description of the Gerstenhaber structure on the homology $H(\fHGC_{1,n})\cong HH(\Poiss_n)$.
Surprisingly, the answer is not quite so trivial and deeply touches upon the subject of deformation quantization.
In fact, the solution requires us to first solve some open (but fortunately not very hard) problems in infinite dimensional deformation quantization. 

Let us give a quick overview. 
The main objective in deformation quantization is the study of formal local deformations of the associative algebra structure on the smooth functions $C^\infty(M)$ on a manifold $M$. Such deformations are governed by the multidifferential Hochschild complex (dg Lie algebra) $\Dpoly(M)\subset CH(C^\infty(M))$, consisting of maps $C^\infty(M)^{\otimes \bullet}\to C^\infty(M)$ which are differential operators in each slot. Maxim Kontsevich's Formality Theorem ``solves'' the deformation quantization problem by describing the homotopy type of the dg Lie algebra $\Dpoly(M)[1]$.

\begin{thm}[Kontsevich Formality Theorem {\cite{K1}} ]\label{thm:Kformality}
 The dg Lie algebra of multidifferential operators $\Dpoly(M)[1]$ is formal, i.e., quasi-isomorphic to its homology, the dg Lie algebra of multivector fields $\Tpoly(M)[1]=C^\infty(M;\La TM)[1]$, equipped with the Schouten bracket.
\end{thm}
To simplify matters, we will from now on set $M=\R^d$, and, if no confusion regarding $d$ arises, abbreviate $\Dpoly=\Dpoly(\R^d)$. 
Two variants of the formality problem can be considered. First, one may consider the sub-dg Lie algebra $\Dpoly^{\geq 1}\subset \Dpoly$ consisting of differential operators with at least one ``slot''. (In other words $\Dpoly=\Dpoly^{\geq 1}\oplus C^{\infty}(\R^d)$.) Clearly $H^{\geq 1}(\Dpoly)\cong \Tpoly^{\geq 1}\subset \Tpoly$ can be identified with the multivector fields of degree at least one as a Lie algebra. Hence it is sensible to ask whether $\Dpoly^{\geq 1}$ is still formal.

Secondly, for applications in quantum field theory one is interested in the case $d= \infty$, and one may ask whether Theorem \ref{thm:Kformality} remains true in this case.

Surprisingly, the answer to both questions is no, as has been shown in \cite[section 6]{DTT} for the first question and \cite{Shoikhet}\footnote{More precisely, in \cite{Shoikhet} it is shown that there is no formality morphism given by ``universal formulas''. It is still an open problem of determining the formality of $\Dpoly(\R^\infty)$ in the strict sense. However, given that one is most interested in ``sensible'' formulas, we still consider the result of \cite{Shoikhet} as a negative answer to our question.} for the second.
Even more surpisingly, the answers are related as we shall now describe.

We will follow the obstruction theoretic viewpoint. Generally, any dg Lie algebra $\alg g$ is quasi-isomorphic to $H(\alg g)$ endowed with \emph{some} $L_\infty$ structure via homotopy transfer. One approach to studying formality questions is hence studying the space of $L_\infty$ structures on $H(\alg g)$. Such structures may be defined as the Maurer-Cartan elements in the Chevalley complex $C(H(\alg g),H(\alg g))$, considered as a dg Lie algebra. 
Consider next the case $H(\alg g)=\Tpoly[1]$. We will cheat a bit and consider not the dg Lie algebra $C(\Tpoly[1],\Tpoly[1])$, but ``universal maps'' that can be build just by taking derivatives and contracting indices. 
More concretely, pick cordinates $\{x_j\}$ on $\R^d$, so that one may identify 
\[
 \Tpoly \cong C^{\infty}(\R^d)[p_1, \dots, p_d]
\]
where $p_j=\frac{\partial}{\partial x_j}$ is a variable of degree 1. Then general degree $r$ multivector fields have the form $\gamma = \gamma^{i_1\dots i_r}(x_1,\dots, x_r) p_{i_1}\cdots p_{i_r}$.
We want to build multilinear maps of those objects by taking derivatives and contracting indices of tensors.
Such derivative and contraction patterns are easily seen to be encoded in a directed graph, like the following:
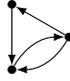
\[
\begin{tikzpicture}
 \node[int] (v1) at (0:.5){};
 \node[int] (v2) at (120:.5){};
 \node[int] (v3) at (240:.5){};
 \draw[-latex] (v1) edge (v2) edge[bend left] (v3) (v2) edge (v3) (v3) edge[bend left] (v1);
\end{tikzpicture}
\]
More concretely, to such a graph with $k$ vertices one can associate a map 
\begin{gather}\label{equ:graphaction}
 S^k(\Tpoly[2]) \to \Tpoly[2] \\
 (\gamma_1,\dots, \gamma_k) \mapsto \mu \sum_{\sigma\in S_k} \pm \prod_{(i,j)\text{ edge}}\sum_{\alpha=1}^d \left(\frac{\partial}{\partial x_\alpha}\right)_{\sigma(i)}
 \left(\frac{\partial}{\partial p_\alpha}\right)_{\sigma(j)}  (\gamma_1,\dots, \gamma_k),
\end{gather}
where the derivative operations act on the $\sigma(i)$-th (resp. $\sigma(j)$-th) factors, and $\mu$ multiplies all factors in the tensor product.

One can build a graded vector space $\dGC_2$ of formal linear combinations of directed graphs like the above (the precise definition will be given below). The formula \eqref{equ:graphaction} then provides us with a map of this graph complex into the Chevalley complex of the Lie algebra of multivector fields
\[
 \dGC_2 \to C(\Tpoly[1],\Tpoly[1]).
\]
Furthermore, there is a natural dg Lie algebra structure on $\dGC_2$ such that the above map is a morphism of dg Lie algebras. The dg algebra $\dGC_2$ is our formal avatar of the Chevalley complex $C(\Tpoly[1],\Tpoly[1])$, and we are interested in the Maurer-Cartan elements of $\dGC_2$, which yield (via the map \eqref{equ:graphaction}) $L_\infty$ structures on $\Tpoly[1]$ given ``by universal formulas'', i.e., independent of the dimension $d$.\footnote{The above replacement of the Chevalley complex by the graph complex $\dGC_2$ is ad hoc. A better justification is discussed in \cite{WillStable}.}

A similar construction works for $\Tpoly^{\geq 1}$ and in the case $d=\infty$. In the case of $\Tpoly^{\geq 1}$ one however has to require that the graphs have a sink,\footnote{I.e., at least one vertex with only incoming arrows.} so that no polyvector fields of degree 0 (i.e., functions) can be produced in \eqref{equ:graphaction}. In the case $d=\infty$ one similarly restricts to graphs that have no directed cycles to avoid potentially infinite summations. Of course, every finite directed acyclic graph has a sink, so that we find the following hierarchy of graphical dg Lie algebras
\[
 \GCor_2 \subset \dGC^{sink}_2\subset \dGC_2,
\]
with $\GCor_2$ spanned by directed acyclic graphs and $\dGC^{sink}_2$ by graphs with a sink.
The Maurer-Cartan elements in $\GCor_2$ have been completely classified. There is, up to gauge, only a one parameter family of such elements.
\begin{thm}[{\cite{WillInfty}}]\label{thm:oriented MC}
 There is a one parameter family of Maurer-Cartan elements $m_\lambda$ in $\GCor_2$, and a gauge-invariant function 
 \[
  F: \MC(\GCor_2)\to \K
 \]
such that any Maurer-Cartan element $m$ is gauge equivalent to $m_{F(m)}$.
\end{thm}
We claim that a similar statement holds for the Maurer-Cartan elements in $\dGC^{sink}_2$ because the inclusion $\GCor_2\to \dGC_2^{sink}$ is a quasi-isomorphism, but we will not show this statement in the current paper.

Now, to cut the story a bit short, the Maurer-Cartan element $m_{trans}\in\GCor_2$ corresponding to the homotopy-transferred $L_\infty$ structure on $\Tpoly$ (or $\Tpoly^{\geq 1}$) can be seen to satisfy $F(m_{trans})\neq 0$ and hence to \emph{not} be trivial, i.e., not correspond to the standard Lie algebra structure on $\Tpoly$.
This is the formal version of non-formality of $\Dpoly^{\geq 1}$, and at the same time the obstruction to existence of infinite dimensional deformation quantization.
We will call the Maurer-Cartan element $m_{trans}$ the \emph{Shoikhet} Maurer-Cartan element, because it was first explicitly constructed by Shoikhet in \cite{Shoikhet}. Similarly, we will call the non-standard $L_\infty$ structure on $\Tpoly^{\geq 1}[1]$ defined by $m_{trans}$ the \emph{Shoikhet} $L_\infty$ structure.

\medskip

Next let us describe the link to the Hochschild homology of $\Poiss_n$. We note that there is a close formal analogy between the main players in deformation quantization, and those for $HH(\Poiss_n)$, as given in the following table.
\bigskip
\begin{center}
 \begin{tabular}{c|c}
  Deformation Quant. & Long knots \\
  \hline
  $\Dpoly^{\geq 1}$ & $CH(\Poiss_n)$ \\
  $\Tpoly^{\geq 1}$ & hairy graph complex $\fHGC_{1,n}$ \\
  Hochschild-Kostant-Rosenberg Theorem & Theorem  \ref{thm:PoisHKR}
 \end{tabular}
\end{center}
\bigskip

Mathematically, we will materialize this formal analogy in this paper by showing that there is a map of dg Lie algebras
\[
 \GCor_2 \to C(\fHGC_{1,n}, \fHGC_{1,n})
\]
by a graphical analog of formula \eqref{equ:graphaction}.
In particular, non-zero Maurer-Cartan elements of $\GCor_2$ yield non-standard $L_\infty$ structures on $\fHGC_{1,n}$, and we have the following result.

\begin{thm}\label{thm:main}
The $L_\infty$ algebra $\fHGC_{1,n}$ with $L_\infty$ structure induced by the Shoikhet Maurer-Cartan element $m_{trans}$ is quasi-isomorphic to $CH(\Poiss_n)$.
\end{thm}
We will call this $L_\infty$ algebra structure on $\fHGC_{1,n}$ the Shoikhet $L_\infty$ structure. We note that in particular, the description above gives us a quite explicit description of the induced Lie bracket on $H(\fHGC_{1,n})$, which will be examined below.

The second main result of this paper in some sense will be the extension of the above statements to the homotopy Gerstenhaber (or, equivalently, homotopy braces) setting.

\begin{thm}\label{thm:brstructure}
The above $L_\infty$ algebra structure on $\fHGC_{1,n}$ may be extended to a homotopy braces (and homotoy Gerstenhaber) structure on $\fHGC_{1,n}$ given by hair reconnection operations, such that $\fHGC_{1,n}$ with that structure is quasi-isomorphic to
$CH(\Poiss_n)$ as homotopy braces algebra. 
\end{thm}

We want to remark that the main point of this paper is not that some such homotopy structures exists on $\fHGC_{1,n}$, which is clear by transfer and Theorem \ref{thm:PoisHKR}, but that one may find somewhat explicit formulas for these structures, and that furthermore, curiously, those formulas are provided by applying techniques and results from deformation quantization.

Finally, let us note that the techniques and results developed in this paper have several interesting implications and applications that might prove useful in future work:

\begin{itemize}
\item Through the action of the directed acyclic graph complexes $\GCor_2$ on the hairy graph complex $\fHGC_{1,n}$, we in particular obtain additional algebraic operations on $\fHGC_{1,n}$ that may be used to analyze its (still unknown) homology. These operations are described in more detail in section \ref{sec:natural ops} below.
\item The computation of the rational homotopy type of the mapping spaces of the $E_n$ operads in \cite{FTW} crucially uses Theorem \ref{thm:shoikhet_projection} of this paper, which is in turn a consequence of Theorem \ref{thm:brstructure}. 
\item During the proof of Theorem \ref{thm:brstructure} we in particular construct a homotopy braces formality morphism in the infinite dimensional setting, extending the construction of \cite{WillInfty} in finite dimensions. This result might be of independent interest to people in deformation quantization.
\end{itemize}

\subsection*{Structure of the paper}
Sections \ref{sec:notation}-\ref{sec:defq} contain mostly recollections about previous results this paper draws upon, in particular, the definition of Hochschild homology for multiplicative operads, and the constructions of formality morphisms by Kontsevich and Shoikhet in deformation quantization.

The key section is Section \ref{sec:action and proofs}, where the action of the directed acyclic graph operads and complexes on the hairy graph complexes is described.
We will work in slightly higher generality than necessary, including also the hairy graph complexes $\fHGC_{m,n}$.
The main results Theorems \ref{thm:main} and \ref{thm:brstructure} are then easy Corollaries shown in the same section.

The remaining sections contain some applications. 
Additional algebraic structures on the hairy graph complexes arising from our construction are discussed in section \ref{sec:natural ops}. 
The induced Lie bracket on the hairy graph homology is studied in section \ref{sec:lie bracket}, and the induced multiplicative structure (i.e., the cup product) in section \ref{sec:cup product}.

\subsection*{Acknowledgements}
This paper is a spin-off and prerequisite of my collaboration \cite{FTW} with B. Fresse and V. Turchin. I am very grateful to Benoit and Victor for discussions and their helpful comments.

\section{Notation and Hochschild homology of multiplicative operads}\label{sec:notation}
We generally work over a ground field $\K$ of characteristic zero. For some integral formulas below we have to require that $\K=\R$, but this will be noted. We will use the language of operads, and refer the reader to the textbook \cite{LVbook} for an introduction to the subject. In particular, we denote the associative operad by $\Ass$, and the $n$-Poisson operad, generated by two binary operations in degrees $0$ and $1-n$, by $\Pois_n$.
We will abbreviate
\[
 \e_n=\begin{cases}
       \Ass & \text{for $n=1$} \\
       \Pois_n & \text{for $n\geq 2$}
      \end{cases}.
\]

Furthermore, we will need to consider cofibrant (or quasi-free) resolutions of common operads. We will generally denote minimal such resolution by the prefix $\mathsf{ho}$. For example, we denote by $\hoLie=\Omega(\Lie^\vee)$ the minimal resolution of $\Lie$, i.e., the operadic cobar construction of the Koszul dual cooperad of the Lie operad.

\subsection{Multiplicative operads and Hochschild complex}
An operad $\op P$ with a map $\Ass\to \op P$ is called a multiplicative operad. For such an operad, the total space 
\[
CH(\op P) = \prod_r \op P(r)[r]
\]
carries a natural differential $d$ (the Hochschild differential) defined as follows. Temporarily denote by $m \in \op P(2)$ the image of the generator of $\Ass$. Then, for $x\in \op P(r)[r]$
\[
dx = \sum_{i=0}^{r+1} d_i x
\] 
with
\begin{align*}
d_0 x &= m\circ_2 x \\
 d_{r+1} x &= m\circ_1(x) \\
d_i x &= x\circ_i m \quad\quad\text{for $1\leq i\leq r $}.
\end{align*}
We call $(CH(\op P), d)$ the Hochschild complex of the multiplicative operad $\op P$. The Hochschild homology $HH(\op P)$ is the homology of the Hochschild complex. 
The usual Hochschild complex of algebras is a special case. If $A$ is a vector space, then an algebra structure on $A$ is the same as an operad map $\Ass\to \End(A)$, so that in particular $\End(A)$ is a multiplicative operad. The Hochschild complex $CH(\End(A))$ is then the standard cohomological Hochschild complex of the algebra $A$ with values in $A$.

The Hochschild complex $(CH(\op P), d)$ carries a natural action of the braces operad $\Br$, given by the same formulas as the braces action for the usual Hochschild complex of algebras, cf. \cite{GJ}. Concretely, for $x\in \op P(r)[r]$ and $x_1,\dots,x_k\in CH(\op P)$ we define the generating operations to be
\begin{align*}
x\{x_1,\dots,x_k\} &= \sum_{1\leq i_1<i_2<\cdots <i_k} \pm x\circ_{i_1,i_2,\dots ,i_k} (x_1,\dots,x_k) \\
x_1 \cup x_2 &= m\circ (x_1,x_2).
\end{align*}
Here, in the second line $m$ is again the image of the generator of $\Ass$ in $\op P$.
By the map $\Lie_1\to \Br$ the Hochschild complex $CH(\op P)$ in particular carries a Lie bracket. the Gerstenhaber bracket. Picking a quasi-isomorphism between a resolution of $\Poiss_2$ and $\Br$, this Lie algebra structure may be extended to a homotopy $\Poiss_2$ algebra structure, depending on the quasi-isomorphism picked. In any case, the Hochschild homology $HH(\op P)$ carries a natural $\Poiss_2$ algebra structure.

\begin{ex}
A relevant example for us will be the following. The multi-differential operators on a manifold $M$ naturally form an operad. It is a multiplicative operad, by sending the image of the generator to the obvious 2-differential operator 
\begin{gather*}
m : C^\infty(M)\otimes C^\infty(M) \to C^\infty(M) \\
(f,g) \mapsto fg.
\end{gather*}
We denote the associated Hochschild complex by $\Dpoly(M)$.
\end{ex}

\section{Graph operads and graph complexes}\label{sec:graphs}

In this section we will introduce several operads and Lie algebras whose elements are linear combinations or series of combinatorial graphs.

\subsection{Graph operads}
Let $\dgra_{r,k}$ denote the set of directed graphs with $r$ numbered vertices and $k$ numbered edges.
We define a vector space 
\[
\dGra_m(r) = \bigoplus_{k} \left( \K\langle \dgra_{r,k}\rangle  \otimes (\K[m-1])^{\otimes k} \right)_{S_k}
\]
as the coinvariants under the symmetric group action.
The spaces $\dGra_m(r)$ naturally assemble to an operad which we call $\dGra_m$. The operadic composition is by inserting a graph into a vertex of another, and summing over all ways of reconnecting the pending edges, see \cite{WillOriented} for more details.
The operad $\dGra_m$ naturally acts on the space of polynomials
\[
\K[x_1,\dots,x_d,p_1,\dots,p_d]
\]
where the $x_j$ are variables of degree 0 and the $p_j$ are variables of degree $m-1$. Concretely the action of a directed graph $\Gamma$ with $r$ vertices on functions $f_1,\dots,f_r$ is defined as
\begin{equation}\label{equ:dGraaction}
\Gamma (f_1,\dots,f_r)=\mu_r \left( \prod_{(i,j)\in E\Gamma} \sum_{\alpha=1}^d \left(\frac{\partial }{\partial p_\alpha}\right)_i\left(\frac{\partial }{\partial x_\alpha}\right)_j \right)(f_1\otimes\cdots \otimes f_r) .
\end{equation}
Note that this action naturally extends to smooth functions on the degree shifted cotangent bundle $C^\infty(T^*[m-1]\R^d)$, and in particular to the multivector fields $\Tpoly\cong C^\infty(T^*[1]\R^d)$ if $m=2$.

There is a sub-operad $\dGraor_m\subset \dGra_m$ spanned by the directed acyclic graphs. Its action on the polynomials $\K[x_1,\dots,x_d,p_1,\dots,p_d]$ in particular preserves the ideal of functions vanishing at $p_1=p_2=\dots=p_d=0$, which we denote by 
\[
p\K[x_1,\dots,x_d,p_1,\dots,p_d].
\]
Similarly, note that for $m=2$, $\dGraor_2$ preserves the multivector fields $\Tpoly^{\geq 1}$ of degree at least one.

Next we want to define a colored variant, and an colored operad of Kontsevich graphs.
We allow ourselves to be a bit brief, since the construction has been explained in detail in \cite{Dol} and \cite[section 3]{WillInfty}.

Denote by $\kgra_{r,k,s}$ ("Kontsevich graphs") the set of directed graphs with $r+s$ vertices numbered by $\{1,\dots,r,\bar 1,\dots, \bar s\}$ and $k$ numbered edges, such that no edges emanate from the vertices $\{1,\dots,r,\bar 1,\dots, \bar s\}$. We call the vertices $\{1,\dots,r\}$ the type I vertices and the vertices $\{\bar 1,\dots, \bar s\}$ the type two vertices. Here is an example of a graph in $\kgra_{r,k,s}$, with the numbering of the edges suppressed for simplicity. Note that we draw the type II vertices on a line in the following, instead of providing the numbering.
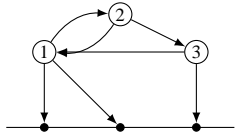
\[
\begin{tikzpicture}
 \draw(-.5,0)--(2.5,0);
 \node[ext] (v1) at (0,1) {$\scriptstyle 1$};
  \node[ext] (v2) at (1,1.5) {$\scriptstyle 2$};
  \node[ext] (v3) at (2,1) {$\scriptstyle 3$};
  \node[int] (w1) at (0,0) {};
  \node[int] (w2) at (1,0) {};
  \node[int] (w3) at (2,0) {};
  \draw[-latex] (v1) edge (w1) edge (w2) edge[bend left] (v2);
  \draw[-latex] (v2) edge[bend left] (v1) edge (v3);
  \draw[-latex] (v3) edge (v1) edge (w3);
\end{tikzpicture}
\]
Define the graded vector spaces
\[
\KGra(r) = \bigoplus_{k,s}\left( \K\langle \kgra_{r,k,s}\rangle  \otimes (\K[1])^{\otimes k} (\K[-1])^{\otimes s} \right)_{S_k}.
\]
The vector spaces $\KGra$ from a right operadic $\dGra_2$-module. Furthermore, they form an operadic left module under the braces operad $\Br$, i.e., an operadic $\Br-\dGra_2$-bimodule. In particular, the $\Br$ structure can be used to endow $\KGra$ with a differential defined as the Lie bracket (part of the $\Br$-operations) with the element 
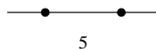
\[
\begin{tikzpicture}
 \draw(-.5,0)--(1.5,0);
  \node[int] (w1) at (0,0) {};
  \node[int] (w2) at (1,0) {};
\end{tikzpicture}
\]
of $\KGra(0)$ of arity 0.
Furthermore, the representations of $\Br$ on $\Dpoly$ and $\dGra_2$ on $\Tpoly$ may be extended to include the operadic bimodule $\KGra$, equipped with the differential above. Concretely, this means that to each graph with $r$ type I vertices one has to assign a map 
\[
\Tpoly^{\otimes r} \to \Dpoly.
\] 
This may be done by a formula analogous to \eqref{equ:dGraaction}.
Overall, one finds that one has an action of the two-colored operad generated by $\KGra$,
\[
\bpm \Br & \KGra & \dGra_2 \epm
\]
on the two-colored vector space 
\[
\Dpoly \oplus \Tpoly.
\]
We again refer to \cite[section 3]{WillInfty} for more details.
The operadic $\Br-\dGra_2$-bimodule $\KGra$ gives rise to a $\Br-\dGraor_2$-bimodule $\KGraor\subset\KGra$ spanned by the graphs without directed cycles, where we consider graphs with vertices without outgoing edges to be zero. The colored operad 
\[
\bpm \Br & \KGraor & \dGraor_2 \epm
\]
it generates acts on the two-colored vector space 
\[
\Dpoly^{\geq 1} \oplus \Tpoly^{\geq 1}.
\]

\subsection{Undirected variant}
There are several similar variants of the graph operads above. For example, the operad $\dGra_m$ contains a sub-operad
\[
 \Gra_m \subset \dGra_m
\]
that consists of linear combinations of graphs (anti-)invariant under flipping edge directions.
This means that elements of $\Gra_m(r)$ may be depicted as linear combinations of undirected graphs with $r$ vertices, e.g., 
\[
 \begin{tikzpicture}
  \node[ext] (v1) at (0,0) {$\scriptstyle 1$};
  \node[ext] (v2) at (1,0) {$\scriptstyle 2$};
  \node[ext] (v3) at (.5,.7) {$\scriptstyle 3$};
  \draw (v2) edge (v1) edge (v3) (v1) edge (v3);
 \end{tikzpicture}.
\]
The inclusion into the operad of directed graphs is then realized by summing over all ways of assigning egde directions, with appropriate signs. Formally, this means that each undirected edge is replaced by the linear combination of edges
\[
 \begin{tikzpicture}[baseline=-.65ex]
  \draw[-latex] (0,0) -- (1,0);
 \end{tikzpicture}
 +(-1)^m
 \begin{tikzpicture}[baseline=-.65ex]
  \draw[latex-] (0,0) -- (1,0);
 \end{tikzpicture}
 .
\]

\subsection{Variant with external legs}
We will also consider a slightly larger variant of the oriented graph operads $\dGra_m$
\[
 \vGra_m \supset \Gra_m.
\]
Concretely, we define $\vGraor_m(r)$ to consist of $\K$-linear series of graphs with ``external legs'', e.g.,
\begin{equation}\label{equ:vGraorsample}
 \begin{tikzpicture}
   \node[ext] (v1) at (0,0) {$\scriptstyle 1$};
   \node[ext] (v2) at (1,0) {$\scriptstyle 2$};
   \node[ext] (v3) at (0,1) {$\scriptstyle 3$};
   \node[ext] (v4) at (1,1) {$\scriptstyle 4$};
   \draw[-latex] (v1) edge (v2) edge (v3) edge (v4) (v2) edge (v3) edge (v4) (v3) edge (v4) 
   (v4) edge +(.5,.5) edge +(0,.5) edge +(.5,0)
   (v1) edge +(-.5,-.5);
 \end{tikzpicture}.
\end{equation}

The operadic compositions $\Gamma_1\circ j\Gamma_2$ are defined similarly to that on $\dGra_m$, except that one attaches outgoing edges at vertex $j$ of $\Gamma_1$ only to external legs in $\Gamma_2$. In particular, if the number of outgoing edges at vertex $j$ is not equal to the number of external legs present in $\Gamma_2$, then the composition is defined to be zero.
The operad $\vGraor_m$ acts on the space of polynomials
\[
\K[x_1,\dots,x_d,p_1,\dots,p_d].
\]
The inclusion $\Gra_m\to \vGra_m$ is realized by the map 
\[
 \Gamma\mapsto \sum_{j\geq 0} \frac{1}{j!} \underbrace{
 \begin{tikzpicture}[baseline=-.65ex]
     \node[ext] (v1) at (0,.5) {$\scriptstyle \Gamma$};
     \node at (0,.1) {$\scriptstyle \dots$};
     \draw[-latex] (v1) edge +(-.5,-.5) edge +(-.75,-.5) edge +(.5,-.5) edge +(.75,-.5);
 \end{tikzpicture}
 }_{j\times}
\]
attaching external legs to $\Gamma$ in all possible ways.
Evidently, the right action of $\dGra_2$ on $\KGra$ factors through $\vGra_2$.

Finally, the operad $\vGra_m$ contains a subquotient $\vGraor_m$ consisting of series of directed \emph{acyclic} graphs only, and in which we furthermore consider graphs with vertices without outgoing edges to be zero. (For example, the directed graph \eqref{equ:vGraorsample} above describes a non-zero element of $\vGraor_m(4)$.
The operad $\Graor_m$ embeds into $\vGraor_m$ by the formula
\[
 \Gamma\mapsto \sum_{j\geq 1} \frac{1}{j!} \underbrace{
 \begin{tikzpicture}[baseline=-.65ex]
     \node[ext] (v1) at (0,.5) {$\scriptstyle \Gamma$};
     \node at (0,.1) {$\scriptstyle \dots$};
     \draw[-latex] (v1) edge +(-.5,-.5) edge +(-.75,-.5) edge +(.5,-.5) edge +(.75,-.5);
 \end{tikzpicture}
 }_{j\times}.
\]
Furthermore, the action of $\Graor_2$ on $\KGraor$ factors through $\vGraor_m$.
In particular, one has the inclusions of two colored operads 
\[
\bpm \Br & \KGraor & \Graor_2 \epm \subset \bpm \Br & \KGraor & \vGraor_2 \epm.
\]

\subsection{Graph complexes}
The graph operads above come with a natural map of operads
\[
\Pois_{m} \to \dGraor_{m} \to \dGra_{m}
\]
defined on the generating operations by
\begin{align}\label{equ:generatormap}
- \wedge - &\mapsto 
\begin{tikzpicture}[baseline=-.65ex]
\node[ext] (v) at (0,0) {1};
\node[ext] (w) at (0.7,0) {2};
\end{tikzpicture}
&
[-,-] &\mapsto 
\begin{tikzpicture}[baseline=-.65ex]
\node[ext] (v) at (0,0) {1};
\node[ext] (w) at (1,0) {2};
\draw[-latex] (v) edge (w) ;
\end{tikzpicture}
+
\begin{tikzpicture}[baseline=-.65ex]
\node[ext] (v) at (0,0) {1};
\node[ext] (w) at (1,0) {2};
\draw[-latex] (w) edge (v) ;
\end{tikzpicture}\, .
\end{align}

The full graph complexes are defined as deformation complexes 
\begin{align*}
\fGCor_m &= \Def(\hoLie_m \to \Graor_m) 
&
\fdGC_m &= \Def(\hoLie_m \to \dGra_m). 
\end{align*}
Both contain subcomplexes of connected graphs $\GCor_m\subset \fGCor_m$ and $\dGC_m\subset \fdGC_m$. For more details on these definitions we refer the reader to \cite{WillOriented}.
The main Theorem relating these two complexes is the following
\begin{thm}[{\cite{WillOriented}}]\label{thm:GCorGC}
There is an isomorphism of Lie algebras respecting the loop order grading
\[
H(\GCor_{m+1}) \cong H(\GC_m).
\]
\end{thm}
We do not have complete knowledge as to what the graph homology $H(\GC_m)$ is, although much is known by now, see the introduction in \cite{KWZ} for an overview.
Fortunately, the most important classes for us will live in $H^1(\GCor_2)\cong H^1(\GC_1)$. The right-hand side is known to be a one-dimensional space, spanned by the ``theta''-graph
\[
 \begin{tikzpicture}
   \node[int] (v) at (0,0) {};
   \node[int] (w) at (0.7,0) {};
   \draw (v) edge (w) edge[bend left=40] (w) edge[bend right=40] (w);
 \end{tikzpicture}
.
\]
The corresponding class in the directed acyclic graph complex $\GCor_2$ is represented by the linear combination
\begin{equation}\label{equ:GCorclass}
\begin{tikzpicture}[baseline={(current bounding box.center)}, scale=.5, every edge/.style={-latex,draw}]
\node[int] (v1) at (-1,1.5) {};
\node[int] (v2) at (-1,-0.5) {};
\node[int] (v4) at (-2,0.5) {};
\node[int] (v3) at (0,0.5) {};
\draw  (v1) edge (v2);
\draw  (v3) edge (v2);
\draw  (v3) edge (v1);
\draw  (v4) edge (v1);
\draw  (v4) edge (v2);
\end{tikzpicture}
+
2\
\begin{tikzpicture}[baseline={(current bounding box.center)},yshift=.5, scale=.5, every edge/.style={-latex,draw}]
\node[int] (v1) at (-1,1.5) {};
\node[int] (v2) at (-1,-0.5) {};
\node[int] (v4) at (-2,0.5) {};
\node[int] (v3) at (0,0.5) {};
\draw  (v1) edge (v2);
\draw  (v2) edge (v3);
\draw  (v1) edge (v3);
\draw  (v4) edge (v1);
\draw  (v4) edge (v2);
\end{tikzpicture}
+
\begin{tikzpicture}[baseline={(current bounding box.center)}, scale=.5, every edge/.style={-latex,draw}]
\node[int] (v1) at (-1,1.5) {};
\node[int] (v2) at (-1,-0.5) {};
\node[int] (v4) at (-2,0.5) {};
\node[int] (v3) at (0,0.5) {};
\draw  (v1) edge (v2);
\draw  (v2) edge (v3);
\draw  (v1) edge (v3);
\draw  (v1) edge (v4);
\draw  (v2) edge (v4);
\end{tikzpicture}
.
\end{equation}

Clearly, through the maps $\Gra_m\to \vGra_m$ and $\Gra_m\to \vGraor_m$ we obtain inclusions on $\Pois$ in those larger graph complexes. In particular, we may consider the graph complexes 
\begin{align*}
\fvGCor_m &= \Def(\hoLie_m \to \vGraor_m) 
&
\fvGC_m &= \Def(\hoLie_m \to \vGra_m),
\end{align*}
along with the subcomplexes of connected graphs $\vGCor_m\subset \fvGCor_m$ and $\vGC_m\subset \fvGC_m$.

These complexes are highly related, we state here only the following result of \cite{WillOriented}.

\begin{prop}[Proposition 3 of \cite{WillOriented}]\label{prop:GChGC}
 The map $\GCor_m \to \vGCor_m$ is a quasi-isomorphism up to the class in $H(\vGCor_m)$ represented by the graph cocycle
\begin{equation}
\label{equ:singleclass}
 \sum_{j\geq 2} \frac{j-1}{j!}
 \underbrace{\begin{tikzpicture}[baseline=-2.5ex, scale=.7]
  \node[int] (v) at (0,0) {};
 \coordinate (v0) at (-.7,-1);
  \coordinate (v1) at (-.3,-1);
  \coordinate (v2) at (.3,-1);
  \coordinate (v3) at (.7,-1);
  \draw[dedge] (v) edge (v0) edge (v1) edge (v2) edge (v3);  
 \end{tikzpicture}
 }_{j\times}.
\end{equation}
\end{prop}

Let us also remark that the above additional class \eqref{equ:singleclass} lives in degree $0$ and does not give additional freedom in the choice of Maurer-Cartan elements in the complexes $\vGCor_m$. In particular, the following analog of Theorem \ref{thm:oriented MC} holds also for the graph complexes with external legs.

\begin{cor}\label{cor:Lieunique}
There is a one parameter family of operad maps
\[
 \phi_\lambda \colon \hoLie_2 \to \vGra_2
\]
corresponding to the MC elements $m_\lambda$ of Theorem \ref{thm:oriented MC}.
Furthermore, on the space of operad maps $\phi:\hoLie_2 \to \vGra_2$ deforming the standard map (factoring through $\Lie_2$)
there is a homotopy invariant function $F$ such that any such $\phi$ is homotopic to $\phi_{F(\phi)}$.
\end{cor}

\begin{rem}\label{rem:1in1outMC}
 It will be important below that the Maurer-Cartan elements $m_\lambda$ of Theorem \ref{thm:oriented MC} may be chosen such that all vertices have either at least two incoming edges, or two outgoing edges. In other words, univalent vertices and vertices of the shape
 \[
  \begin{tikzpicture}
   \node[int] (v) at (0,0) {};
   \draw[-latex] (v) edge +(.5,0);
   \draw[latex-] (v) edge +(-.5,0);
  \end{tikzpicture}
 \]
can be forbidden.
To see this, one realizes that the subcomplex of $\GCor_m$ formed by such graphs is a dg Lie subalgebra quasi-isomorphic to the full graph complex, cf. \cite{WillOriented}.
\end{rem}

\subsection{Twisted graph operads}
Given an $L_\infty$ algebra $\alg g$ and a Maurer-Cartan element $m\in \alg g$ we may (setting some convergence issues aside) define the twisted $L_\infty$ algebra $\alg g^m$. Now suppose that the $L_\infty$ algebra structure is part of a larger algebraic structure governed by an operad $\op P$, i.e., one has morphisms
\[
\hoLie\to \op P \to \End(\alg g).
\]
The one can ask whether one can also twist the $\op P$-structure, i.e., whether the map 
\[
\hoLie\to \End(\alg g^m)
\]
can be made to factor through $\op P$. In general, this is not possible. However, one may define a larger operad $\Tw \op P$ with a map $\Tw \op P\to \op P$ such that the above action factors as
\[
\hoLie\to \Tw \op P \to \End(\alg g^m).
\]
The construction (operadic twisting) is developed in more detail in \cite{DolWill}. Explicitly, the operad $\Tw\op P$ is the operad generated by $\op P$ and a zero-ary operation representing the Maurer-Cartan element $m$, up to completion on the number of such elements present.

Applying the twisting construction to the operad $\Gra_m$ we obtain an operad $\Tw\Gra_m$ whose elements are $\K$-linear series of graphs, some of whose vertices are "unidentifiable" (i.e., filled by the formal avatar of the Maurer-Cartan element), like the following graph, where the inputs filled by $m$ have been marked by coloring the corresponding vertices black.
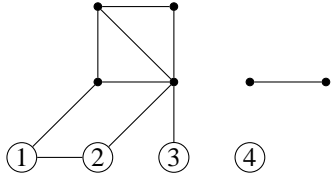
\[
\begin{tikzpicture}
 \node[ext] (v1) at (0,0) {1};
 \node[ext] (v2) at (1,0) {2};
 \node[ext] (v3) at (2,0) {3};
 \node[ext] (v4) at (3,0) {4};
 \node[int] (w1) at (1,1) {};
 \node[int] (w2) at (1,2) {};
 \node[int] (w3) at (2,1) {};
 \node[int] (w4) at (2,2) {};
 \node[int] (w5) at (3,1) {};
 \node[int] (w6) at (4,1) {};
 \draw (v1) edge (v2) edge (w1) 
       (w2) edge (w1) edge (w3) edge (w4)
       (w3) edge (v3) edge (v2) edge (w1) edge (w4)
       (w5) edge (w6);
\end{tikzpicture}
\]
In fact, we will pass to a slightly smaller sub-operad $\Graphs_{m}\subset \Tw\Gra_m$ by requiring that all connected components of graphs must have at least one numbered ("external") vertex.
There is a map $\Graphs_m\to \Gra_m$ by sending graphs with internal vertices to zero, and we have a map $\Pois_m\to \Graphs_m$ given by mapping the generators as follows.
\begin{align*}
- \wedge - &\mapsto 
\begin{tikzpicture}[baseline=-.65ex]
\node[ext] (v) at (0,0) {1};
\node[ext] (w) at (0.7,0) {2};
\end{tikzpicture}
&
[-,-] &\mapsto 
\begin{tikzpicture}[baseline=-.65ex]
\node[ext] (v) at (0,0) {1};
\node[ext] (w) at (1,0) {2};
\draw (v) edge (w) ;
\end{tikzpicture}
\end{align*}
The important fact for us is the following result.
\begin{thm}[Kontsevich \cite{K2}, Lambrechts-Voli\'c \cite{LVformal}]
The map $\Pois_m\to \Graphs_m$ is a quasi-isomorphism of operads.
\end{thm}

In a similar manner we may define the twisted operads $\Tw\dGra_m$ and $\Tw \dGraor_m$.

\subsection{The hairy graph complexes, and the relation to the Hochschild complex of $\Poiss_n$.}\label{sec:hairydef}
Let us briefly recall here the definition of the hairy graph complexes, see \cite{TW} for a more detailed account.
We consider the operadic deformation complexes (dg Lie algebras)
\[
\Def(\hoPoiss_{m}\stackrel{*}{\to} \Graphs_n)
\]
where the map $*$ is the composition 
\[
\hoPoiss_m \to \Poiss_m \to \Com \to \Poiss_n \to  \Graphs_n .
\]
Concretely, the underlying vector space of the deformation complex above has the form 
\[
\prod_r \Hom_{S_r}(\Pois_m^\vee(r), \Graphs_n(r))
\cong \prod_r \Pois_m\{m\}(r),\otimes_{S_r} \Graphs_n(r).
\]
It contains a natural subcomplex
\[
\fHGC_{m,n} \subset \Def(\hoPoiss_{m}\stackrel{*}{\to} \Graphs_n),
\]
the hairy graph complex, consisting of series in
\[
  \prod_r \Com_m\{m\}(r),\otimes_{S_r} \Graphs_n'(r) \subset  \prod_r \Pois_m\{m\}(r),\otimes_{S_r} \Graphs_n(r),
\]
where $\Graphs_n'(r)\subset \Graphs_n(r)$ is the subcomplex consisting of graphs all of whose external vertices have valence exactly one.
Concretely, elements of $\fHGC_{m,n}$ may be depicted as graphs with external legs or hairs as shown in \eqref{equ:HGCsample}.
The important result is the following:
\begin{thm}[\cite{Turchin1}]
The inclusion $\fHGC_{m,n} \subset \Def(\hoPoiss_{m}\stackrel{*}{\to} \Graphs_n)$ is a quasi-isomorphism of dg Lie algebras.
\end{thm}

Concretely, the Lie bracket on hairy graphs is combinatorially given by attaching a hair of one graph to a vertex of the other, as schematically depicted here.
\[
\left[ 
\begin{tikzpicture}[baseline=-.8ex]
\node[draw,circle] (v) at (0,.3) {$\Gamma$};
\draw (v) edge +(-.5,-.7) edge +(0,-.7) edge +(.5,-.7);
\end{tikzpicture}
,
\begin{tikzpicture}[baseline=-.65ex]
\node[draw,circle] (v) at (0,.3) {$\Gamma'$};
\draw (v) edge +(-.5,-.7) edge +(0,-.7) edge +(.5,-.7);
\end{tikzpicture}
\right]
=
\sum
\begin{tikzpicture}[baseline=-.8ex]
\node[draw,circle] (v) at (0,1) {$\Gamma$};
\node[draw,circle] (w) at (.8,.3) {$\Gamma'$};
\draw (v) edge +(-.5,-.7) edge +(0,-.7) edge (w);
\draw (w) edge +(-.5,-.7) edge +(0,-.7) edge +(.5,-.7);
\end{tikzpicture}
\pm 
\sum
\begin{tikzpicture}[baseline=-.8ex]
\node[draw,circle] (v) at (0,1) {$\Gamma'$};
\node[draw,circle] (w) at (.8,.3) {$\Gamma$};
\draw (v) edge +(-.5,-.7) edge +(0,-.7) edge (w);
\draw (w) edge +(-.5,-.7) edge +(0,-.7) edge +(.5,-.7);
\end{tikzpicture}\, .
\]

For us, one important fact is that the hairy graph complexes $\fHGC_{m,n}$ also compute the Hochschild homology of $\Poiss_n$, namely one has the following result.
\begin{thm}[\cite{Turchin1}]\label{thm:graphsHKR}
The symmetrization maps $\Com(r)\to \Ass(r)$ induce a quasi-isomorphism of complexes (but not of Lie algebras)
\[
\Phi : \fHGC_{1,n} \to  CH(\Graphs_n).
\]
\end{thm}

This morphism is formally similar to the Hochschild-Kostant-Rosenberg map, and hence we will refer to it as the hairy Hochschild-Kostant-Rosenberg morphism.

\subsection{Stable formality morphisms}
Let $\ELie$ be the 2-colored operad governing two $\hoLie_2$ algebras and an $\infty$-morphism between them.
We define a stable formality morphism (cf. \cite{WillInfty,Dol}) to a colored operad map 
\[
\ELie \to \bpm \Br & \KGra & \vGra_2 \epm
\]
extending the usual maps $\hoLie_2\to \Br$ and $\hoLie_2\to \vGra_2$, and such that the induced map of complexes $\Tpoly\to \Dpoly$ coincides with the Hochschild-Kostant-Rosenberg quasi-isomorphism. Concretely, the latter part of the condition means that the linear component of the $L_\infty$ morphism is given by the series of graphs
\begin{equation}\label{equ:HKRelement}
\sum_{k\geq 0}
\frac 1 {k!}
\underbrace{
\begin{tikzpicture}
\draw (-1,0)--(1,0);
\node[ext] (v) at (0,1) {}; 
\node[int] (w1) at (-.8,0) {};
\node[int] (w2) at (-.6,0) {};
\node (wx) at (0,.4) {$ \cdots$};
\node[int] (w3) at (.6,0) {};
\node[int] (w4) at (.8,0) {};
\draw[-latex] (v) edge (w1) edge (w2) edge (w3) edge (w4);
\end{tikzpicture}
}_{k\times}.
\end{equation}

Similarly, an oriented stable formality morphism is a colored operad map
\[
\ELie \to \bpm \Br & \KGraor & \vGraor_2 \epm
\]
extending the usual map $\hoLie_1\to \Br$ (but not $\hoLie_1\to \vGraor_2$), and such that the induced $L_\infty$ map $\Tpoly^{\geq 1}\to \Dpoly^{geq 1}$ coincides with the Hochschild-Kostant-Rosenberg quasi-isomorphism. Concretely, this means that the linear piece of the $L_\infty$ map is given by the following series of graphs.
\begin{equation}\label{equ:HKRelementor}
\sum_{k\geq 1}
\frac 1 {k!}
\underbrace{
\begin{tikzpicture}
\draw (-1,0)--(1,0);
\node[ext] (v) at (0,1) {}; 
\node[int] (w1) at (-.8,0) {};
\node[int] (w2) at (-.6,0) {};
\node (wx) at (0,.4) {$ \cdots$};
\node[int] (w3) at (.6,0) {};
\node[int] (w4) at (.8,0) {};
\draw[-latex] (v) edge (w1) edge (w2) edge (w3) edge (w4);
\end{tikzpicture}
}_{k\times}.
\end{equation}
Note that now the summation starts at $k=1$, since graphs with vertices without outgoing edges are considered zero in $\KGraor$.

Constructing such formality morphisms is not easy.
We shall recall in section \ref{sec:defq} below Kontsevich's construction of a stable formality morphism, and constructions of Shoikhet and the author of an oriented one. Note that part of the data of an oriented formality morphism is a non-standard $L_\infty$ structure on the space of multivector fields.

We will also consider an extension taking into account the full $E_2$ structures. 
The two relevant models of the $E_2$ operad we will be using are the braces operad $\Br$, and the minimal resolution $\hoe_2$ of the operad $\e_2=\Poiss_2$.

To this end define $\EBr$ as the 2-colored operad governing two $\Br_\infty$ algebras and an $\infty$ morphism between them, and similarly $\EGer$ as the two colored operad governing two $\hoe_2$ algebras and an $\infty$-morphism between them. Then a 
stable braces formality morphisms (respectively oriented stable braces formality morphisms) is a colored operad map 
\[
\EBr \to \bpm \Br & \KGra & \vGra_2 \epm
\]
or respectively,
\[
\EBr \to \bpm \Br & \KGraor & \vGraor_2 \epm
\]
extending the natural map $\Br_\infty\to \Br$, and inducing stable formality morphism (resp. oriented stabile formality morpshisms) through the restriction $\Lie_1\to \Br$. Similarly we define stable (oriented) $\hoe_2$-formality morphisms by replacing $\EBr$ by $\EGer$ above. We note that is essentially irrelevant whether to consider homotopy braces or $\hoe_2$ formality morphisms, as one can always compose with (suitable) quasi-isomorphisms
\[
 \EGer \to \EBr\to \EGer
\]
to pass from one type to another.

\begin{rem}
One may also use a more restrictive notion of stable formality morphism re requiring that the colored operad maps above factor through one of the smaller colored operads
\[
 \bpm \Br & \KGra & \Gra_2 \epm \subset  \bpm \Br & \KGra & \dGra_2 \epm \subset  \bpm \Br & \KGra & \vGra_2 \epm
\]
as was for example done originally in \cite{Dol}. In this paper we generally stick to the laxer version above, as the refinement would only introduce spurious additional technical difficulties.
\end{rem}

\subsection{Twistable stable formality morphisms}\label{sec:twistable}
There is also a colored version of the operadic twisting operation, cf. \cite{WillInfty} or the Appendix in \cite{ricardoBr}. In our case, given a stable formality morphism
\[
\ELie\to \bpm \Br & \KGra & \vGra_2 \epm
\]
it produces for us a $\Br$-$\Tw\vGra_2$-operadic bimodule $\Tw\KGra$ together with a map 
\[
\ELie\to \bpm \Br & \Tw\KGra & \Tw\vGra_2 \epm.
\]
These data have the following property: Suppose $\bpm \Br & \KGra & \vGra_2 \epm$ acts on a colored vector space $V\oplus W$, and suppose that $m\in W$ is a Maurer-Cartan element. (With respect to the $\hoLie_1$ structure on $W$ induced by the given map from $\ELie$.) Then the twisted action
\[
\ELie \to \End(V^{f(m)}\oplus W^m),
\]
where $f(m)$ is the image in $W$ of the Maurer-Cartan element $m$, factors through an action
\[
 \bpm \Br & \Tw\KGra & \Tw\dGra_2 \epm \to \End(V^{f(m)}\oplus W^m).
\]

In general, we will call a twistable stable formality morphism a map 
\[
\ELie \to \bpm \Br & \Tw\KGra & \Tw\vGra_2 \epm
\]
such that the composition 
\[
\ELie \to \bpm \Br & \Tw\KGra & \Tw\vGra_2 \epm \to \bpm \Br & \KGra & \vGra_2 \epm
\]
is a stable formality morphism.
Similarly, we define the notion of stable twistable braces formality morphism a map 
\[
\EBr \to \bpm \Br & \Tw\KGra & \Tw\dGra_2 \epm
\]
that yields a stable braces formality morphism after projection to $ \bpm \Br & \KGra & \dGra_2 \epm$. Analogously, we define the notion of twistable $\hoe_2$ formality map.
Yet analogously, we define the notion of twistable stable oriented (braces or $\hoe_2$) formality morphism, by replacing the graph operads on the right-hand side by their acyclic versions.

We note that twistable stable formality morphisms can be obtained from ordinary ones. We consider only the oriented case, and $\hoe_2$ formality morphisms, as this will be used below.
The twisting functor is a comonad \cite{DolWill}, and a co-algebra over this comonad we call a natively twistable operad. This means in particular, that such an operad $\op P$ comes equipped with a map 
\[
 \op P\to \Tw\op P
\]
right inverse to the canonical projection $\Tw \op P\to \op P$. For example, $\hoe_2$ and $\Br$ are natively twistable.
As shown in \cite{DolWill} we may choose a quasi-isomorphism $\hoe_2\to \Br$ compatible with the $\Tw$-coalgebra structure, i.e., such that the following diagram commutes
\begin{equation}\label{equ:Twfactorizing}
 \begin{tikzcd}
  \hoe_2 \ar{r} \ar{d}& \Br\ar{d} \\
  \Tw\hoe_2 \ar{r} & \Tw\Br 
 \end{tikzcd}\ .
\end{equation}
Suppose now we can extend the map above to a stable formality morphism
\[
 \EGer \to \bpm \Br & \KGraor & \vGra_2 \epm.
\]
Applying the twisting functor, and using that $\EGer$ is natively twistable (in the colored setting) we obtain 
the composition
\[
 \EGer \to   \bpm \Tw\Br & \Tw\KGraor & \Tw\vGra_2 \epm.
\]
Now, due to \eqref{equ:Twfactorizing} the morphism in color one (i.e., $\hoe_2\to \Br$) factors through $\Br$, and we thus obtain the desired twistable formality morphism
\[
 \EGer \to   \bpm \Br & \Tw\KGraor & \Tw\vGra_2 \epm.
\]

\begin{rem}
 Again, one may define various more restrictive versions of twistable stable formality morphisms by requiring that they factorize through suboperads 
 \[
 \bpm \Br & \Tw\KGra & \Graphs_2 \epm \subset \bpm \Br & \Tw\KGra & \Tw\Gra_2 \epm \subset \bpm \Br & \Tw\KGra & \Tw\dGra_2 \epm \subset \bpm \Br & \Tw\KGra & \Tw\vGra_2 \epm .
 \]
 We shall stick to the laxer definition, though some of the formality morphisms we construct below will satisfy the more restrictive definition.
\end{rem}

\section{Deformation quantization}\label{sec:defq}
In this section we give brief recollections of the existing constructions of stable formality morphisms, and slight variations thereof.
This section can be considered as background information for the reader's convenience. It can be skipped provided the reader accepts the following Theorem.

\begin{thm}\label{thm:stable formality existence}
There exist stable formality morphisms and oriented stable formality morphisms over the ground fields $\K=\R$ and $\K=\Q$.
There exist stable braces or $\hoe_2$ formality morphisms over the ground field $\K=\R$. 
There exist oriented stable braces or $\hoe_2$ formality morphisms over the ground fields $\K=\R$ and $\K=\Q$.
\end{thm}

The Theorem is a collection of results in the literature about deformation quantization. The remainder of this section contains recollections of the constructions from the literature, with one slight extension, for the readers convenience. 

\begin{rem}
Note that Theorem \ref{thm:stable formality existence} does not state the existence of a stable $E_2$ formality morphism over $\K=\Q$. There is no doubt that such morphisms can be constructed by the standard methods existing in the literature. In particular in the unstable setting Tamarkin \cite{Tamanother} gives a well known construction. However, the similar construction for the stable case is apparently not written down in the literature, and we will not go through the effort here since we do not use this case.
\end{rem}

\subsection{Recollection: The Kontsevich formality theorem}
The Kontsevich formality Theorem \cite{K1} states that there is an $L_\infty$ quasi-isomorphism 
\[
\mU: \Tpoly[1] \to \Dpoly[1]
\]
between the Lie algebra of multivector fields on $\R^d$ and the Lie algebra of multi-differential operators. 
In fact, Kontsevich's morphism can be written in the form
\begin{equation}\label{equ:Kontsevich formality}
 \mU_k(\gamma_1,\dots,\gamma_k) = \sum_\Gamma c_\Gamma D_\Gamma(\gamma_1,\dots,\gamma_k)
\end{equation}
where the sum runs over \emph{Kontsevich graphs}, $c_\Gamma$ is a constant and $D_\Gamma(\cdots)$ is a certain multidifferential operator. The numbers $c_\Gamma$ can be computed through configuration space integrals
\begin{equation}\label{equ:cGamma}
c_\Gamma = \int_{C_{v_I(\Gamma), v_{II}(\Gamma)}} \omega_\Gamma,
\end{equation}
where $C_{p,q}$ is (essentially) the space of configurations of $p$ points in the upper halfplane and $q$ on the real line and $v_{I,II}(\Gamma)$ are the numbers of type I and type II vertices in $\Gamma$. The differential form $\omega_\Gamma$ is defined as a product over 1-forms, one for each edge
\begin{equation}\label{equ:omegadef}
\omega_\Gamma = \bigwedge_{(ij)\text{ edge}} \frac 1 {2pi} d\phi(z_i, z_j)
\end{equation}
where $\phi_{z_i,z_j}$ is the hyperbolic angle between $\infty$ and $z_j$, as measured at $z_i$, see the following picture.
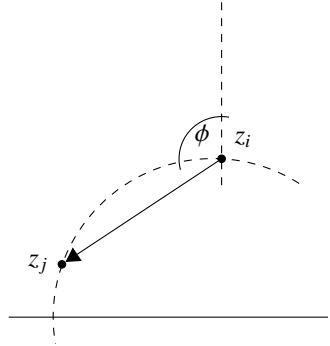
\[
\begin{tikzpicture}[scale=.7]
\draw (-3,0)--(3,0);
\node[int, label=130:{$\phi$}, label=30:{$z_i$}] (v1) at (1,3) {};
\node[int, label=180:{$z_j$}] (v2) at (-2,1) {};
\draw[dashed] (v1)+(0,-.5) -- +(0,3);
\draw (v1)+(80:.8) arc (80:195:.8);
\clip (-5,-.5) rectangle (2.5,4);
\node [draw, dashed] at (.84,0) [circle through={(v1)}] {};
\draw[-triangle 60] (v1)--(v2);
\end{tikzpicture}
\]
The Kontsevich construction determines, in fact, a stable formality morphism, 
\[
\mU = \sum_\Gamma c_\Gamma \Gamma .
\]
The formality morphism is, however, not oriented, as graphs $\Gamma$ with directed cycles are not necessarily associated zero weight $c_\Gamma$. 

A rational stable formality morphism has been constructed in \cite{Dolrational}, using Kontsevich's result.

\subsection{Recollection: The Shoikhet formality theorem}
A definition of a stable oriented formality morphism was given by B. Shoikhet \cite{Shoikhet}.
The construction of the morphism is almost identical to Kontsevich's. The only difference is that in the definition of the integrand \eqref{equ:omegadef}, one replaces $ d\phi(z_i, z_j)$ by 
\begin{equation}\label{equ:dphitrunc}
d\phi^{trunc}(z_i, z_j)=
\begin{cases}
2d\phi(z_i, z_j) & \text{if $\mathrm{Im}( z_i) > \mathrm{Im}( z_j)$} \\
0 & \text{otherwise}.
\end{cases}
\end{equation}
This differential form is supported on configurations in which $z_i$ is above $z_j$. Because of this feature graphs $\Gamma$ with directed cycles are assigned a weight form $\omega_\Gamma=0$, and hence $c_\Gamma=0$, so that the corresponding stable formality morphism is oriented.

However, the "cost" is that (-for reasons we will not describe in detail here-) one has to introduce a nonstandard $L_\infty$ structure on $\Tpoly$. In its stable incarnation this is a non-standard map 
\[
\hoLie_1 \to \dGra_2.
\]
The image of the generator $\mu_n$ is again given by a sum-of graphs formula
\[
\mu_n \mapsto \sum_{\gamma} \tilde c_\gamma \gamma
\]
with $\tilde c_\gamma$ given by a configuration space integral of the form
\[
\tilde c_\gamma = \int_{C_{v(\Gamma)}} \bigwedge_{(ij) \text{ edge}}
\frac{1}{pi} d{\rm arg}^{trunc}(z_i-z_j).
\]
For more details we refer the reader to \cite{Shoikhet}.

\subsection{Recollection: Oriented stable formality morphisms through obstruction theory, and (homotopy) uniqueness}
There is alternative approach to constructing formality morphism. One can use standard obstruction theory and check that obstructions to the existence of stable formality morphisms live in $H^1(\fGC_2)$. In fact, this obstruction theoretic approach was first pursued by Kontsevich \cite{Kconjecture} in attempting to solve his formality conjecture. Unfortunately, the question of whether $H^1(\fGC_2)\stackrel{?}=0$ has been an open problem which resisted the attacks of many researchers over the last decades. In particular, Kontsevich had to resort to the more elaborate (and transcendental) solution of his formality conjecture outlined above.

Fortunately, for stable oriented formality morphisms the situation is slightly different, as has been noted 
in \cite{WillOriented}. 
The obstructions are governed by $H^1(\GCor_2)$. But by Theorem \ref{thm:GCorGC} this space may be identified with $H^1(\GC_2)$, which is a one-dimensional space by essentially trivial degree reasons.
Hence it can obstruction theoretically be shown that oriented stable formality morphism exist, and hence that they in particular exist over $\K=\Q$.
Furthermore, it follows that up to homotopy there is only a one parameter family of stable $L_\infty$ structures, i.e., of maps 
\[
 \hoLie_2\to \dGra_2,
\]
see Theorem \ref{thm:oriented MC} above.
The one parameter is the coefficient of the graphs \eqref{equ:GCorclass}. In particular, it follows that the map $\hoLie_2\to \dGra_2$ constructed by Shoikhet using integrals can be changed to a homotopic rational one.

\subsection{Recollection: The homotopy braces formality theorem}
The $L_\infty$ formality morphisms described above do not take into account the full braces (or, more generally, $E_2$-)structure on the Hochschild complex. A formality morphism taking the $E_2$ structure into account was first constructed by Tamarkin \cite{Tamanother}. However, this construction was not at the ``stable'' level, i.e., factoring through graph operads. 

An extension to a stable (and twistable) homotopy braces morphism has been constructed by the third author in \cite{WillInfty}. 
Cutting the story a bit short, one can build a two colored of configuration spaces 
\[
 T:= \bpm \Br & C(C_{\bullet, 0}) & C(\FM_2) \epm.
\]
Then one can define natural maps 
\begin{equation}\label{equ:Brmorphism}
 \EBr \to T \to \bpm \Br & \KGraphs & \Graphs \epm,
\end{equation}
where $\EBr$ is the two colored operad governing to homotopy braces algebras and an $\infty$-morphism between them.
The right hand map here is defined by configuration space integrals similar to \eqref{equ:cGamma}. 

\subsection{Oriented homotopy braces formality morphism over $\R$}\label{sec:orbrformalityR}
The third author's construction from the previous subsection can be easily modified to yield an oriented twistable stable homotopy braces formality morphism: One merely has to replace the Kontsevich propagator $d\phi$ by its truncated variant \eqref{equ:dphitrunc} considered by Shoikhet. Then, for the same reasons as above, graphs with directed cycles will be associated a zero weight form $\omega_\Gamma$, and will not appear in the image of \eqref{equ:Brmorphism}. Hence we obtain the desired stable braces formality morphism 
\begin{equation}\label{equ:Brmorphism2}
 \EBr \to T \to \bpm \Br & \KGraphsor & \Graphsor \epm.
\end{equation}

We note that since all maps above are fairly explicit, up to the need to compute certain configuration space integrals. 

\subsection{Oriented homotopy $E_2$ formality morphism over $\Q$}
We note that the construction of an oriented stable formality morphism is not "on the same level of difficulty" as the construction of a (non-oriented) stable formality morphism as defined above, despite the fact that the solutions in both cases can be cast in a very similar framework. 
Stable formality morphisms are in general hard to construct and all known constructions use transcendental methods at some point. On the other hand, the existence of stable oriented formality morphisms is just a graphical variant of the homotopy transfer principle. We shall illustrate this in this section and demonstrate the existence of stable oriented $E_2$-formality morphisms over $\K=\Q$.

We will construct an oriented stable $\hoe_2$ formality morphism over $\Q$.
We will in dissect our graphical homotopy transfer argument in two steps for clarity of the exposition. 
\begin{lemma}
There is a morphism of colored operads
\[
\EGer \to \bpm \Br & \KGraor & \vGraor \epm.
\]
such that the linear piece of the morphism agrees with the HKR map.
\end{lemma}
\begin{proof}
The morphism can be obtained by homotopy transfer. Consider the HKR map 
\[
\Tpoly(\R^d) \to \Dpoly(\R^d)
\]
for $d$ very large. We want to homotopy transfer the $\hoe_2$ structure on $\Dpoly$ to $\Tpoly$. To this end we have to pick a projection
\[
\pi: \Dpoly\to \Tpoly \subset \Dpoly,
\]
for which we take the projection to differential operators which are derivations in each "slot", followed by antisymmetrization. Furthermore, we have to pick a homotopy
\[
h: \Dpoly\to \Dpoly[1]
\]
such that $dh+hd=1-\pi$. This homotopy may be picked in such a way that it only "reorders" the derivatives $\frac{\partial}{\partial x}$ that make up the differential operator, not affecting the coefficient of the differential operator. A concrete formula can be extracted from \cite{vergne_lie}.

Now the homotopy transferred $\hoe_2$ structure and the $\hoe_2$ ($\infty$-)map $\Tpoly\to \Dpoly$ are constructed from the above data and the $\hoe_2$ algebra structure just by composition, cf. \cite[section 10.3]{LVbook}. It follows in particular, that the operations obtained are all constructed by taking derivatives of coefficient functions and contracting indices accordingly. But those operations are exactly the operations encoded by the operad $\vGraor$ (for the $\hoe_2$ structure on $\Tpoly$) and $\KGraor$ (for the higher homotopies of the map). Hence the Lemma follows.
\end{proof}

Note that by the uniqueness of the $\hoLie_2$ algebra structure (Corollary \ref{cor:Lieunique}) we may always assume that the map $\hoLie_2\to \vGraor$ is the Shoikhet map (or any fixed representative of its homotopy class).
Indeed, if it is not, there is some homotopy, i.e., an $L_\infty$ isomorphism $\phi: \Tpoly^{(1)}\to \Tpoly^{(2)}$ connecting the two $L_\infty$ algebra, which is stable, i.e., expressible only through graphical formulas (operations in $\vGraor$). We may then alter our $\hoe_2$ formality morphism above by postcomposing with the map $\phi$, and equipping $\Tpoly$ with the pushed-forward $\hoe_2$ structure. Note that all operations are still expressed using graphs, so we obtain an altered stable formality morphism 
\[
\EGer \to \bpm \Br & \KGraor & \vGraor \epm,
\]
and now the map $\hoLie_2\to \vGraor$ is the desired one.
Finally, by operadic twisting as described in section \ref{sec:twistable} we may upgrade our formality morphism to a twistable one.

\section{The action on hairy graphs, and proof of the main theorems}\label{sec:action and proofs}
A priori, one may think that the aforementioned formality Theorems in deformation quantization have little do to with the rational homotopy theory of spaces of long knots.
However, note that given the existence of stable (and twistable) oriented formality morphisms, the only missing step in the proofs of Theorems \ref{thm:main} and \ref{thm:brstructure} is in fact, the following assertion.

\begin{prop}\label{prop:graphs act on graphs}
The 2-colored operad
\[
 \bpm \Br & \Tw\KGraor & \Tw\vGraor_2 \epm
\]
acts on the 2-colored dg vector space
\[
CH(\Graphs_{n}) \oplus \fHGC_{1,n}
\]
such that the induced braces structure on $CH(\Graphs_{n})$, and the induced $\Poiss_2$-structure on $\fHGC$ are the standard ones, and such that the HKR element \eqref{equ:HKRelementor} in $\KGraphsor(1)$ induces the standard HKR map from Theorem \ref{thm:graphsHKR}.
\end{prop}

Let us use this proposition to prove the main Theorems.
\begin{proof}[Proof of Theorems \ref{thm:main} and \ref{thm:brstructure}]
First note that through the quasi-isomorphism $\Poiss_n\to \Graphs_n$ we obtain a quasi-isomorphism $CH(\Poiss_n) \to CH(\Graphs_{n})$ compatible with all structures. We may hence replace $\Poiss_n$ by $\Graphs_n$ in the statements of Theorems \ref{thm:main} and \ref{thm:brstructure}. 

Now pick some twistable stable oriented formality morphism (cf. Theorem \ref{thm:stable formality existence})
\[
\EBr \to \bpm \Br & \Tw\KGraor & \Tw\vGraor_2 \epm.
\]
Composing with the action of the right-hand side on the two-colored vector space $CH(\Graphs_{n}) \oplus \fHGC_{1,n}$ then immediately gives the formality morphism of Theorem \ref{thm:brstructure}. By restriction along $\ELie\to \EBr$ we obtain the formality morphism over $\R$ whose existence is asserted in Theorem \ref{thm:main}.

Over $\K=\R$ we may obtain integral formulas for all structures and morphisms by using the Shoikhet formality morphism and its braces extension of section \ref{sec:orbrformalityR}.
\end{proof}

Now let us turn to the proof of Proposition \ref{prop:graphs act on graphs}. 
For later reference, we will in fact also show the following result, covering the action on $\fHGC_{m,n}$ for any $m$.
\begin{prop}\label{prop:maction}
The operad $\Tw\vGraor_{m+1}$ acts on $\fHGC_{m,n}$, inducing the standard $\Poiss_{m+1}$ structure 
via the canonical map $\Poiss_{m+1}\to \Tw\vGraor_{m+1}$. 
\end{prop}

The action will be defined in two steps.
\begin{enumerate}
\item First we split the differential on $\fHGC_{m,n}$ in two pieces $\delta=\delta_1+\delta_2$.
We then show that the operad $\vGraor_{m+1}$ acts on $(\fHGC_{m,n}, \delta_1)$. Similarly one may split the differential on $\Graphs_n$ and show that 
the colored operad $\bpm \Br & \KGraor & \vGraor_2 \epm$ acts on the colored vector space 
\[
CH((\Graphs_{n}, \delta_1))  \oplus (\fHGC_{1,n},\delta_1).
\]
\item The yet missing piece of the differential $\delta_2$ is obtained as the Lie bracket with a Maurer-Cartan element in $\fHGC_{m,n}$. We hence can apply the formalism of operadic twisting in a second step to obtain the desired action of $\Tw\vGraor_{m+1}$ on $(\HGC_{m,n}, \delta)$, and similarly the action of $\bpm \Br & \Tw\KGraor & \Tw\vGraor_2 \epm$ on $CH(\Graphs_{n})  \oplus (\fHGC_{1,n},\delta)$.
\end{enumerate}

Concretely, the action of $\vGraor_{m+1}$ on the graded vector space $\HGC_{m,n}$ is defined as follows. 
Consider a directed acyclic graph $\Gamma\in \dGraor_{m+1}$ with $k$ vertices and $k$ hairy graphs $\gamma_1,\dots,\gamma_k\in \HGC_{m,n}$. 
We denote the set of hairs of $\gamma_i$ by $H\gamma_i$ and the set of vertices by $V\gamma_i$. We assume that the number of hairs of $\gamma_i$ is equal to the number of outgoing edges at vertex $i$ of $\Gamma$, and define the action to be zero otherwise.
Furthermore, we suppose that $\Gamma$ has $p$ edges $(i_1,j_1),\dots,(i_p,j_p)$. 
Under this condition the action produces the following linear combination of graphs 
\[
 \Gamma(\gamma_1,\dots,\gamma_k) := \prod_{h_1\in H\gamma_{i_1}, v_1\in V\gamma_{j_1} }\cdots \prod_{h_p\in H\gamma_{i_p}, v_p\in V\gamma_{j_p} }
 \pm \gamma(h_1\to v_1,\dots,h_p\to v_p)
\]
where $\gamma(h_1\to v_1,\dots,h_p\to v_p)$ is the graph obtained from the union of the $\gamma_j$'s by connecting hair $h_i$ to vertex $v_i$, for all $i$.
If in the product any hair is chosen twice, we set $\gamma(h_1\to v_1,\dots,h_p\to v_p)=0$ by convention. 
The not reconnected hairs become the hairs of the output.

The action of the colored operad $\bpm \Br & \KGraor & \vGraor_2 \epm$ is defined by analogous formulas.

The action just constructed is (not yet) compatible with the differentials.
More precisely, the differential on $\HGC_{m,n}$ may be constructed in two steps: (i) there is a piece $\delta_1$ of the differential inserting the graph
$\begin{tikzpicture}[scale=.5, baseline=-.65ex]
 \node[int] (v) at (0,0) {};
 \node[int] (w) at (0.5,0) {};
 \draw (v) -- (w);
 \end{tikzpicture}
$
into each internal vertex. 
(ii) The full differential is then 
\[
\delta= \delta_1 + 
\underbrace{[\mu
 ,-]}_{\delta_2}.
\]
where 
\[
\mu:=
 \begin{tikzpicture}[baseline=-.65ex]
  \node[int] (v) at (0,0.2){};
  \draw (v) -- +(0,-.5);
 \end{tikzpicture}.
\]
The piece $\delta_1$ of the differential is already compatible with the action just constructed. The piece $\delta_2=[\mu,-]$ can be accounted for by operadic twisting. For the procedure to be applicable, we need that the element $\mu$ is a Maurer-Cartan element, and we need that the differential $\delta$ is indeed obtained by twisting with $\mu$. These statements are more or less obvious if $\HGC_{m,n}$ is equipped with the standard $\hoLie_{m+1}$-structure, i.e., the structure induced by the map $\hoLie_{m+1}\to \Lie_{m+1}\to \Graor_{m+1}$ sending the Lie bracket generator as in \eqref{equ:generatormap}. However, in the case $m=1$ we have to use the Shoikhet $\hoLie_2$-structure instead, induced by the Shoikhet map $\hoLie_{m+1}\to \Graor_{m+1}$. For this structure the aforementioned statements are less obvious, whence we formulate them as a Lemma.

\begin{lemma}
The element $\mu:=
 \begin{tikzpicture}[baseline=-.65ex]
  \node[int] (v) at (0,0.2){};
  \draw (v) -- +(0,-.5);
 \end{tikzpicture}$ 
 is a Maurer-Cartan element with respect to the Shoikhet $\hoLie_{m+1}$ structure on $(\HGC_{m,n},\delta_1)$ given by the Shoikhet map $\hoLie_{2}\to \Graor_{2}$ and the $\Graor_{m+1}$ action on $(\HGC_{1,n},\delta_1)$ defined above. Furthermore, the differential on $\HGC_{1,n}$ obtained by twisting by $\mu$ is indeed $\delta$.
\end{lemma} 
\begin{proof}
The element $\mu$ has exactly one hair. Hence the only graphs that can contribute to the Maurer-Cartan equation must be such that every vertex has at most one outgoing edge. However, the only directed acyclic such graphs are trees. Trees with more than 2 vertices are of the wrong cohomological degree to appear in the MC equation as one easily counts. But the tree with exactly two vertices just produces Lie bracket of the standard Lie structure on $\HGC_{1,n}$, for which $\mu$ is a Maurer-Cartan element. This shows the first statement, i.e., that $\mu$ is a Maurer-Cartan element for the Shoikhet $\hoLie_2$-structure on $\HGC_{1,n}$.

To show the second statement first recall our standing assumption (or assertion) that the image of the Shoikhet map $\hoLie_{2}\to \Graor_{2}$ does not involve graphs with vertices with exactly one incoming and exactly one outgoing vertex, cf. Remark \ref{rem:1in1outMC}. Now the directed acyclic graphs contributing to the twisted differential must have the following characteristic: (i) all vertices except for possibly one have $\leq 1$ outgoing edge, (ii) no vertex has exactly one incoming and one outgoing edge, (iii) to have the correct cohomological degree the graph must have $k$ vertices and $2k-3$ edges for some $k$. We claim that the only graph satisfying these conditions is again the graph with two vertices connected by one edge, which produces the standard (non-deformed) contribution to the twisted differential. Indeed, a directed acyclic graph satisfying (i) must have the following shape: after deleting one vertex (call it "special vertex") and all edges adjacent to this vertex the remaining graph must be a forest. But to satisfy (iii) the forest must in fact be a tree and the special vertex must connect to all vertices of the tree. But then vertices at the leaves of the tree have one incoming and one outgoing edge, contradicting (ii), unless the whole graph contains only two vertices. Hence the Lemma is shown.
\end{proof}

Knowing that $\mu$ is a Maurer-Cartan element in $(\fHGC_{m,n},\delta_1)$, including the case $m=1$, the "machinery" of operadic twisting \cite{DolWill} immediately gives us an action of the operad $\Tw\vGraor_{m+1}$ on the twisted vector space $(\HGC_{m,n},\delta)$. Concretely, this action is obtained by "inserting" a copy of the Maurer-Cartan element in the "black" vertices occupied by the formal avatar of the Maurer-Cartan element in graphs in $\Graphsor_{m+1}$.

Similarly, by the colored version of twisting one obtains an action of the colored operad  
$\bpm \Br & \Tw\KGraor & \Tw\vGraor_2 \epm$ on $CH(\Graphs_n)  \oplus (\fHGC_{1,n},\delta)$.
Again, the action is constructed by inserting $\mu$ in the slots occupied by the formal avatar of the Maurer-Cartan element ("black vertices") in graphs in $\KGraphsor $ and $\Graphsor_2$. Note that in general only few of these graphs will contribute, namely those for which each black vertex has at most one outgoing edge.

There is remaining technical point to be checked here.
We already checked that twisting the $\hoLie_2$ structure on $\fHGC[1]$ yields the standard differential. We should also check that twisting the $\hoLie_2$ structure on $CH(\Graphs_n)$ produces the standard differential.

\begin{lemma}
The image of the MC element $\mu$ under the Shoikhet $L_\infty$ morphism is again $\mu$, interpreted as an element of $CH(\Graphs_n)$. In particular, the twist reproduces the standard differential on $CH(\Graphs_n)$.
\end{lemma}
\begin{proof}
The contributing acyclic Kontsevich graphs must be such that all type I vertices have exactly one outgoing vertex. If the graph contains at least two type one vertices or at least three vertices, then all type I vertices must have valence at least two, and all type II vertices valence at least I. Furthermore, we can again assume that there are no vertices with exactly one incoming and exactly one outgoing edge. But now the top type I vertex has no incoming edges and exactly one outgoing edge. Hence there must be exactly one other vertex of type II, to which it connects by an edge. But this graph just produces $\mu$, considered as element of $CH(\Graphs_n)$.
\end{proof}

\begin{rem}
 We note that while we have constructed an action of the graph complex $\Tw\vGraor_{m+1}$ on $\fHGC_{m,n}$, we will only really need in explicit calculations below the action of the much smaller sub-operad $\Tw\Graor_{m+1}$.
\end{rem}



\subsection{The projection to the connected part}
Consider $\fHGC_{m,n}$ with the standard $\Poiss_{m+1}$-structure. It is clear that the connected piece $\HGC_{m,n}\subset \fHGC_{m,n}$ is a sub-dg $\Lie_{m+1}$ algebra. Also, it is clear that the canonical projection to the connected piece $\fHGC_{m,n}\to \HGC_{m,n}$ is a map of $\Lie_{m+1}$ algebras, left inverse to the inclusion.

Now consider $\fHGC_{1,n}$ with the Shoikhet $\hoLie_{2}$ algebra structure.
Since this structure is expressed using connected graphs in $\Graor_2$ only, it is clear that the connected piece $\HGC_{1,n}\subset \fHGC_{1,n}$ is a again a sub-$\hoLie_2$-algebra.
However, in this case it is a priori not clear whether the inclusion of the connected piece has an $L_\infty$ left (homotopy) inverse $\fHGC_{m,n}\to \HGC_{m,n}$.
The goal of this section is to show the following Theorem, giving an affirmative answer.
\begin{thm}\label{thm:shoikhet_projection}
 There is an $L_\infty$ morphism
 \[
  \fHGC_{1,n}' \to \HGC_{1,n}'
 \]
between the full and the connected hairy graph complexes, endowed with the Shoikhet $L_\infty$ structure, such that the linear piece of the $L_\infty$ morphism is just the obvious projection sending disconnected graphs to zero. This statement holds over the ground field $\K=\Q$.
\end{thm}

We use the assertion of Theorem \ref{thm:stable formality existence} that the $L_\infty$ structure on $\fHGC_{1,n}'$ is part of a $\hoe_2$ structure on
 $\fHGC_{1,n}'[1]$ defined over $\K=\Q$. Hence we can use the following result.

\begin{lemma}\label{lem:ginftyprojection}
Let $\alg g$ be a $\hoe_2$ algebra. Then the bar construction as a $\hoCom$ algebra $B_{\Com}(\alg g)$ is endowed with a canonical $L_\infty$ structure and one has a natural $L_\infty$ map
\[
\alg g\dashrightarrow B_{\Com}(\alg g).
\]
\end{lemma}
\begin{proof}
The first assertion is well known. In fact $B_{\Com}(\alg g)$ is an infinity Lie bialgebra, but we care only about the $L_\infty$ part of the structure. Now the desired $L_\infty$ map is realized by the following natural zigzag
\[
\alg g \stackrel{\sim}{\dashleftarrow}  
\Omega_{\e_2^\vee}(B_{\e_2}(\alg g)) \to \Omega_{\Com}(B_{\e_2}\alg g)\cong  \stackrel{\sim}{\dashleftarrow} B_{\Com}(\alg g).
\]
\end{proof}

Furthermore, in our case $\alg g=\fHGC_{1,n}'$ is in fact a free commutative algebra in the primitive (connected) part $\HGC_{1,n}'$. As a consequence one can show the following result.
\begin{lemma}\label{lem:connectedprojection}
The composition of maps of $L_\infty$ algebras
\[
\HGC_{1,n}'\hookrightarrow \fHGC_{1,n}' \dashrightarrow B_{\Com}(\fHGC_{1,n}').
\]
is a quasi-isomorphism.
\end{lemma}
\begin{proof}
The statement only involves the linear component of the $L_\infty$ map. The linear component of the right-hand map can be realized just as the inclusion of complexes. Hence the combined map of complexes 
\[
\HGC_{1,n}'\to B_{\Com}(\fHGC_{1,n}').
\]
is just the natural embedding. The homology of the right-hand side may be computed by the spectral sequence associated to the filtration by the length of Lie words. On the $E^1$ page we find the free Lie coalgebra in $H(\fHGC_{1,n}')$. The differential is the Harrison differential with respect to the commutative product structure. But since $H(\fHGC_{1,n}')\cong S^+(H(\HGC_{1,n}'))$ is free, this homology is given by the generators $H(\HGC_{1,n}')$. Hence on the $E^2$ page the map above is a quasi-isomorphism, and hence a quasi-isomorphism by standard spectral sequence arguments.
\end{proof}

By Lemmas \ref{lem:ginftyprojection} and \ref{lem:connectedprojection} we hence have an $L_\infty$ map 
\[
\fHGC_{1,n}'\dashrightarrow \HGC_{1,n}'
\]
whose linear component is just the projection to the connected part, and Theorem \ref{thm:shoikhet_projection} is shown.

\hfill\qed

\subsection{Remark: Standard Kontsevich formality and an extended hairy graph complex}
Note that we were forced into discussing the acyclic (non-)formality because, in some sense, our hairy graph complexes $\fHGC_{m,n}$ are too small in that they are not allowed to contain graphs without hairs. If we artificially enlarge these complexes by allowing graphs without hairs, i.e., we replace $\fHGC_{m,n}=S^+(\HGC_{m,n})$ by $C:= S^+(\HGC_{m,n}\oplus \GC_n[m-n])$, then many of the statements simplify. Concretely, the space $C$ is acted upon by the operad $\dGra_2$. As a consequence the ordinary Lie algebra structure on $\fHGC_{1,n}$ naturally extends, as well as the Shoikhet $L_\infty$ structure given by acyclic graphs as discussed above, and both $L_\infty$ algebras are quasi-isomorphic
\[
 S \stackrel{\sim}{\dashrightarrow} S'.
\]
Since the Shoikhet $L_\infty$ structure is given by connect graphs, the result also descends to the connected parts, i.e., there is an $L_\infty$ quasi-isomorphism
\[
 \HGC_{1,n}\oplus \GC_n[1-n] \stackrel{\sim}{\dashrightarrow} \left( \HGC_{1,n}\oplus \GC_n[1-n] \right)'.
\]

Unfortunately, we currently do not know what the meaning of the enlarged dg Lie algebra $\HGC_{1,n}\oplus \GC_n[1-n]$ is from a topological standpoint.

\subsection{Remark: The $\Poiss_m$ homology of $\Poiss_n$}
In this paper, we discuss the Hochschild complexes of the Poisson operads $\Poiss_n$, i.e., the deformation complexes
\[
\Def(\hoe_1\to \Poiss_n).
\]
Similarly, one can also consider the "higher" Hochschild complexes $\Def(\hoe_m\to \Poiss_n)$. These complexes carry a $\hoPoiss_{m+1}$-algebra structure, and their homology is quasi-isomorphic to the hairy graph complex $\fHGC_{m,n}$. 
However, in the case $m>1$ the relation between the $\hoe_{m+1}$-algebra $\Def(\hoe_m\to \Poiss_n)$ and the $\Poiss_{m+1}$-algebra $\fHGC_{m,n}$ is much simpler: The natural inclusion (quasi-isomorphism) 
\[
\fHGC_{m,n}\to \Def(\hoe_m\to \Poiss_n)
\]
is already compatible with the $\hoPoiss_{m+1}$-structure, cf. \cite{CalWill} and also \cite[section 2.1]{TW}.

\section{Operad of natural operations on $\fHGC_{m,n}$}\label{sec:natural ops}
We have seen above that the operad $\Tw \Graor_{m+1}$ acts on the hairy graph complex $\fHGC_{m,n}$, and that the action factors through the quotient operad 
\[
\Grap^{or}_{m+1} \leftarrow \Tw \Gra_{m+1}^{or}
\]
defined by sending graphs to zero that contain internal vertices with more than one outgoing edge.
We consider the operad $\Grap^{or}_{m+1}$ as the operad of natural operations (given by reconnecting hairs) on $\HGC_{m,n}$. In particular, there is an action of the homology $H(\Grap^{or}_{m+1})$ on the hairy graph homology $\HGC_{m,n}$. We cannot compute $H(\Grap^{or}_{m+1})$ explicitly, but we note that the homology operad is quite large and contains many interesting operations on the hairy graph homology. We merely give a list of such operations here.

\begin{itemize}
\item The simplest operations are those of the standard $\Poiss_{m+1}$-structure, given by the maps
\[
\Pois_{m+1} \to \Tw \Gra_{m+1}^{or} \to \Grap^{or}_{m+1}
\]
defined on generators by
\begin{align*}
- \wedge - &\mapsto 
\begin{tikzpicture}[baseline=-.65ex]
\node[ext] (v) at (0,0) {1};
\node[ext] (w) at (0.7,0) {2};
\end{tikzpicture}
&
[-,-] &\mapsto 
\begin{tikzpicture}[baseline=-.65ex]
\node[ext] (v) at (0,0) {1};
\node[ext] (w) at (1,0) {2};
\draw[-latex] (v) edge (w) ;
\end{tikzpicture}
+
\begin{tikzpicture}[baseline=-.65ex]
\node[ext] (v) at (0,0) {1};
\node[ext] (w) at (1,0) {2};
\draw[-latex] (w) edge (v) ;
\end{tikzpicture}\, .
\end{align*}
\item The total symmetric space of $H(\Grap^{or}_{m+1})$ contains the graph cohomology $H(\fGC^{or}_{m+1})\cong H(\fGC^2_{m})$. To see this note that one has natural maps of complexes
\[
\fGC^{or}_{m+1} \to \Def(\hoLie_{m+1}\to \Grap^{or}_{m+1}) \to \fGC^{or}_{m+1}.
\]
The first map sends a graph to the sum of graphs obtained by coloring arbitrary subsets of vertices black, and the second all graphs with black vertices to 0. The composition of both maps is the identity, hence in particular the first map is an injection on homology. To give one example, from the graph homology class \eqref{equ:GCorclass} one can find the non-trivial homology class in $H(\Grap^{or}_{m+1})$ represented by
\[
\begin{tikzpicture}[baseline=-.65ex]
\node[ext] (v) at (0,0) {1};
\node[ext] (w) at (1,0) {2};
\node[int] (x) at (.5,.6) {};
\node[int] (y) at (.5,1.1) {};
\draw[-latex] (v) edge (x) edge (y) (w) edge (x) edge (y) (x) edge (y) ;
\end{tikzpicture}
+ 2\
\begin{tikzpicture}[baseline=-.65ex]
\node[ext] (v) at (0,0) {1};
\node[ext] (w) at (1,0) {2};
\node[int] (x) at (.5,.6) {};
\node[int] (y) at (.5,1.1) {};
\draw[-latex] (v) edge (x) edge (w) (w) edge (x) edge (y) (x) edge (y) ;
\end{tikzpicture}
+
\begin{tikzpicture}[baseline=-.65ex]
\node[ext] (v) at (0,0) {1};
\node[ext] (w) at (1,0) {2};
\node[int] (x) at (0,1) {};
\node[int] (y) at (1,1) {};
\draw[-latex] (v) edge (x) edge (y) edge (w) (w) edge (x) edge (y) ;
\end{tikzpicture}
+
(1\leftrightarrow 2)\, .
\]
\item There are more operations in $H(\Grap^{or}_{m+1})$. In particular, for $m$ even we find the following homology class in $H(\Grap^{or}_{m+1}(1))$
\[
\begin{tikzpicture}[baseline=-.65ex]
\node[ext] (v) at (0,0) {1};
\node[int] (x) at (0,1) {};
\draw[-latex] (v) edge[bend left] (x) edge[bend right] (x) ;
\end{tikzpicture}
\]
For $m=0$ this class is of degree 1 and represents an infinitesimal deformation of the differential. The corresponding full (i.e., non-infinitesimal) deformation is given by the Maurer-Cartan element 
\[
\sum_{k=2}^\infty \frac{1}{k!}  
\underbrace{
\begin{tikzpicture}[baseline=2ex]
\node[ext] (v) at (0,0) {1};
\node[int] (x) at (0,1) {};
\node at (0,.5){$\scriptstyle \cdots$};
\draw[-latex] (v) edge[bend left] (x) edge[bend left=60] (x) edge[bend right=60] (x) edge[bend right] (x);
\end{tikzpicture}
}_{k \text{ edges}}
\]
in the dg algebra $\Grap^{or}_{m+1}(1)$. This deformation of the differential has been used in \cite{KWZ2} to obtain a significant amount of information about the hairy graph homology for $m$ even. 
\end{itemize}

\section{Examples of the Lie bracket on homology}\label{sec:lie bracket}
An(y) $L_\infty$ structure on the hairy graph complex $\HGC_{1,n}$ induces a Lie algebra structure on the hairy graph homology $H(\HGC_{1,n})$. The goal of this section is to compute this Lie bracket for some homology classes, as an application of the theory developed above. 

In general, the Lie bracket may receive contributions from graphs in the Shoikhet MC element $m_{trans}$ of arbitrarily high loop order. For our example calculations, we will however only consider (and mostly only need) the piece of loop order $\leq 2$:
\[
m_{trans} = 
\begin{tikzpicture}[baseline={(current bounding box.center)}, scale=.5, every edge/.style={-latex,draw}]
\node[int] (v1) at (-1,1.5) {};
\node[int] (v2) at (-1,-0.5) {};
\node[int] (v4) at (-2,0.5) {};
\node[int] (v3) at (0,0.5) {};
\draw  (v1) edge (v2);
\draw  (v3) edge (v2);
\draw  (v3) edge (v1);
\draw  (v4) edge (v1);
\draw  (v4) edge (v2);
\end{tikzpicture}
+
2\
\begin{tikzpicture}[baseline={(current bounding box.center)},yshift=.5, scale=.5, every edge/.style={-latex,draw}]
\node[int] (v1) at (-1,1.5) {};
\node[int] (v2) at (-1,-0.5) {};
\node[int] (v4) at (-2,0.5) {};
\node[int] (v3) at (0,0.5) {};
\draw  (v1) edge (v2);
\draw  (v2) edge (v3);
\draw  (v1) edge (v3);
\draw  (v4) edge (v1);
\draw  (v4) edge (v2);
\end{tikzpicture}
+
\begin{tikzpicture}[baseline={(current bounding box.center)}, scale=.5, every edge/.style={-latex,draw}]
\node[int] (v1) at (-1,1.5) {};
\node[int] (v2) at (-1,-0.5) {};
\node[int] (v4) at (-2,0.5) {};
\node[int] (v3) at (0,0.5) {};
\draw  (v1) edge (v2);
\draw  (v2) edge (v3);
\draw  (v1) edge (v3);
\draw  (v1) edge (v4);
\draw  (v2) edge (v4);
\end{tikzpicture}
+\text{higher loop order}.
\]
We obtain the Lie bracket of two graphs $\gamma_1,\gamma_2$ by twisting. Concretely, this means that the inputs corresponding to two vertices are $\gamma_1,\gamma_2$, while to all other vertices, we have to assign the MC element 
$
\begin{tikzpicture}
 \node[int] (v) at (0,0) {};
 \draw (v) edge +(.5,0);
\end{tikzpicture}
$, discarding the graph if that other vertex has out-valence $>1$. Depicting the MC element with black vertices, and $\gamma_1,\gamma_2$ by white vertices, we hence obtain the following three contributing terms:
\begin{equation}\label{equ:Liebracket}
\begin{tikzpicture}[baseline={(current bounding box.center)}, scale=.5, every edge/.style={-latex,draw}]
\node[int] (v1) at (-1,1.5) {};
\node[int] (v2) at (-1,-0.5) {};
\node[ext] (v4) at (-2,0.5) {};
\node[ext] (v3) at (0,0.5) {};
\draw  (v1) edge (v2);
\draw  (v3) edge (v2);
\draw  (v3) edge (v1);
\draw  (v4) edge (v1);
\draw  (v4) edge (v2);
\end{tikzpicture}
+
2\
\begin{tikzpicture}[baseline={(current bounding box.center)},yshift=.5, scale=.5, every edge/.style={-latex,draw}]
\node[ext] (v1) at (-1,1.5) {};
\node[int] (v2) at (-1,-0.5) {};
\node[ext] (v4) at (-2,0.5) {};
\node[int] (v3) at (0,0.5) {};
\draw  (v1) edge (v2);
\draw  (v2) edge (v3);
\draw  (v1) edge (v3);
\draw  (v4) edge (v1);
\draw  (v4) edge (v2);
\end{tikzpicture}
+
\begin{tikzpicture}[baseline={(current bounding box.center)}, scale=.5, every edge/.style={-latex,draw}]
\node[ext] (v1) at (-1,1.5) {};
\node[ext] (v2) at (-1,-0.5) {};
\node[int] (v4) at (-2,0.5) {};
\node[int] (v3) at (0,0.5) {};
\draw  (v1) edge (v2);
\draw  (v2) edge (v3);
\draw  (v1) edge (v3);
\draw  (v1) edge (v4);
\draw  (v2) edge (v4);
\end{tikzpicture}
 +\text{(higher loop orders)}
\end{equation}

In the following we will restrict to the case of $\HGC_{1,n}$ for odd $n$, because the lowest non-trivial cohomology classes are simpler (i.e., have fewer vertices) than in the case of even $n$, and thus our computations are simpler.

\subsection{Loop order 0}
The only non-trivial homology classes in $\HGC_{1,n}$ for $n$ odd of loop order 0 are $\K$-multiples of the line graph
\[
L=
\begin{tikzpicture}[baseline=-.65ex]
\draw (0,0) -- (1,0);
\end{tikzpicture}.
\]
Only the first term in \eqref{equ:Liebracket} can contribute non-trivially. The second cannot contribute because the line graph has no vertices, so if a white vertex in \eqref{equ:Liebracket} is hit by an arrow, the term vanishes. The third term in \eqref{equ:Liebracket} can similarly not contribute because the line graph has only two hairs, but there is a vertex with three outgoing edges. Finally the terms of higher loop order necessarily also contain such vertices by degree reasons and hence yield no contribution.
Overall we find that
\[
[L,L ] =
 \begin{tikzpicture}[baseline=-.65ex]
   \node[int] (v) at (0,0) {};
   \node[int] (w) at (0.7,0) {};
   \draw (v) edge (w) edge[bend left=40] (w) edge[bend right=40] (w)
         (w) edge +(.5,0);
 \end{tikzpicture}
\]
This is a non-trivial class. 
On the other hand, the standard Lie bracket of $L$ with itself vanishes.
Even more strikingly, $[L,L]$ cannot be produced by the standard Lie bracket.
Hence we can state the following corollary.
\begin{lemma}
The Shoikhet Lie algebra structure on $H(\HGC_{1,n})$ is nontrivial and differs from the standard Lie algebra structure.
Furthermore there is no isomorphism between $H(\HGC_{1,n})$ with the standard and the Shoikhet Lie algebra structures such that for $x$ of loop order $g$
\[
x \mapsto x + \text{(terms of loop order $>g$)}.
\]
\end{lemma}


\subsection{Loop orders 0 and 1}
The simplest class of loop order one in $H(HGC_{1,n})$ for $n$ odd is represented by the 2-hedgehog graph
\[
H_2=
\begin{tikzpicture}[baseline=-.65ex]
\node[int] (v) at (0,0) {};
\node[int] (w) at (1,0) {};
\draw (v) edge[bend left] (w) edge[bend right] (w)
      edge +(-.5,0) (w) edge +(.5,0);
\end{tikzpicture}
\]
The standard Lie bracket $[H,L]_{std}=0$ is trivial. The Shoikhet Lie bracket receives corrections (only) from the first two terms of \eqref{equ:Liebracket}. These terms read
\[
[H_2,L]_{Shoi} = 
\begin{tikzpicture}[baseline=-.65ex]
 \node[int] (v1) at (-.4,-.4) {};
 \node[int] (v2) at (-.4,.4) {};
 \node[int] (v3) at (.4,-.4) {};
 \node[int] (v4) at (.4,.4) {};
 \draw (v1) edge[bend left] (v2) edge[bend right](v2) edge (v3)
       (v4) edge[bend left] (v3) edge[bend right](v3) edge (v2)
       (v3) edge +(0,-.5);
\end{tikzpicture}
+2\
\underbrace{
\begin{tikzpicture}[baseline=-.65ex]
 \node[int] (v1) at (-.4,-.4) {};
 \node[int] (v2) at (-.4,.4) {};
 \node[int] (v3) at (.4,-.4) {};
 \node[int] (v4) at (.4,.4) {};
 \draw (v1) edge[bend left] (v2) edge[bend right](v2) edge (v3)
       (v4) edge (v3) edge (v1) edge (v2)
       (v3) edge +(0,-.5);
\end{tikzpicture}
}_{\text{exact term}}
= \frac{1}{3}
\begin{tikzpicture}[baseline=-.65ex]
 \node[int] (v1) at (-.4,-.4) {};
 \node[int] (v2) at (-.4,.4) {};
 \node[int] (v3) at (.4,-.4) {};
 \node[int] (v4) at (.4,.4) {};
 \draw (v1) edge (v2) edge (v4) edge (v3)
       (v4) edge (v3) edge (v2)
       (v3) edge +(0,-.5) edge (v2);
\end{tikzpicture}
+ \text{(exact terms)}.
\]
We note that the right-hand side is shown to represent a non-trivial homology class in \cite{TW2}.

\subsection{Loop orders 1 and 1}
Similarly, one can compute brackets of graphs of loop order 1.
For example, we leave it to the reader to check that 
\begin{align*}
[H_2,H_2]_{std} &= 
\begin{tikzpicture}[baseline=-.65ex]
 \node[int] (v) at (0,0) {};
 \node[int] (w) at (.7,0) {};
 \node[int] (x) at (1,0) {};
 \node[int] (y) at (1.7,0) {};
 \draw (v) edge +(-.5,0) edge[bend left](w) edge[bend right](w)
       (w) edge (x)
       (x) edge[bend left] (y) edge[bend right](y) edge +(0,-.5)
       (y) edge +(0.5,0);
\end{tikzpicture}
=
2
\begin{tikzpicture}[baseline=-.65ex]
 \node[int] (v) at (0,0) {};
 \node[int] (w) at (.5,.5) {};
 \node[int] (x) at (1,0) {};
 \node[int] (y) at (.5,-.5) {};
 \draw (v) edge +(-.5,0) edge (w) edge (y)
       (w) edge (x) edge (y)
       (x) edge (y) edge (w) edge +(0.5,0)
       (y) edge +(0,-.5);
\end{tikzpicture}
+(\text{exact terms})
 \neq 0 \text{ (in homology)}
\end{align*}
The non-triviality of the right-hand side of the first line has been shown in \cite{CCT}, where the corresponding class is denoted by $XXX$.
Note that this is also the first case when the standard bracket of two graphs is not trivial.
In particular, the standard Lie algebra structure is not trivial. 

The Shoikhet bracket in this case again retrieves additional contributions (only) from the first two terms in \eqref{equ:Liebracket}. We find that
$$
[H_2,H_2]_{Shoi} = [H_2,H_2]_{std}
+ 
\begin{tikzpicture}[baseline=-.65ex]
 \node[int] (v) at (0,-.4) {};
 \node[int] (w) at (0,.4) {};
 \node[int] (x) at (.7,-.4) {};
 \node[int] (y) at (.7,.4) {};
 \node[int] (z) at (1.4,-.4) {};
 \node[int] (a) at (1.4,.4) {};
 \draw (v) edge[bend left] (w) edge[bend right] (w)        
       (x) edge (y) edge (v) edge (z) edge +(0,-0.5)
       (y) edge (w) edge (a)
       (z) edge[bend left] (a) edge[bend right](a) ;
\end{tikzpicture}
+ 2\
\begin{tikzpicture}[baseline=-.65ex]
 \node[int] (v) at (0,-.4) {};
 \node[int] (w) at (0,.4) {};
 \node[int] (x) at (.7,-.4) {};
 \node[int] (y) at (.7,.4) {};
 \node[int] (z) at (1.4,-.4) {};
 \node[int] (a) at (1.4,.4) {};
 \draw (v) edge[bend left] (w) edge[bend right] (w)        
       (x) edge (y) edge (v) edge (z) 
       (y) edge (w) edge[bend left] (a) edge[bend right] (a)
       (z) edge (a)  edge +(0,-0.5);
\end{tikzpicture}
+ 2\
\begin{tikzpicture}[baseline=-.65ex]
 \node[int] (v) at (0,-.4) {};
 \node[int] (w) at (0,.4) {};
 \node[int] (x) at (.7,-.4) {};
 \node[int] (y) at (.7,.4) {};
 \node[int] (z) at (1.4,-.4) {};
 \node[int] (a) at (1.4,.4) {};
 \draw (v) edge[bend left] (w) edge[bend right] (w)        
       (x) edge (y) edge (v) edge (z) 
       (y) edge[bend left] (a) edge[bend right] (a)
       (z) edge (a)  edge +(0,-0.5)
       (w) edge[bend left=60] (a);
\end{tikzpicture}
\ .
$$

Unfortunately, more complicated examples quickly become combinatorially very complicated.
For computing the Shoikhet bracket of graphs with more hairs one needs to know higher orders of the Shoikhet MC element $m_{trans}$, and no closed formula is known. For example, for computing  the Shoikhet bracket, of, say, the 4-hedgehog
  \[
  H_4=
  \begin{tikzpicture}[baseline=-.65ex]
  \node[int](v) at (0,0.5) {};
  \node[int](w) at (0,-0.5) {};
  \node[int](x) at (1,-0.5) {};
  \node[int](y) at (1,0.5) {};
  \draw (v) edge (w) edge (y) edge +(-0.5,0.5)
        (w) edge (x) edge +(-0.5,-0.5)
        (x) edge (y) edge +(.5, -.5)
        (y) edge +(.5,.5);
 \end{tikzpicture} 
 \]
 with $H_2$ one a priori needs to know the $4$-loop corrections to \eqref{equ:Liebracket}, and the number of graphs involved in the computations quickly grows.
  
\section{The cup product on Hochschild homology}\label{sec:cup product}
Above, we have mostly focused on studying the combinatorial form of the Lie Gerstenhaber bracket on $HH(\Poiss_n)$, when expressed through hairy graphs. We have not said much about the (cup) product so far. One reason is that the cup product is comparatively simple (but not trivial), and can be derived directly from knowledge of the HKR morphism of Theorem \ref{thm:PoisHKR}. 

First note that, clearly, the cup product will be given by a certain degree zero operation in
\[
m_2 \in \dGraphsor_2(2).
\]
Degree zero is the top degree in that space, elements of degree zero can be identified with bivalent forests of internal vertices, whose leaves are affixed to one of the external vertices, as depicted in the following example.
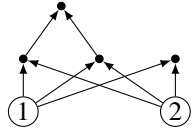
\[
\begin{tikzpicture}[baseline=-.65ex]
\node[ext](v1) at (0,0) {1};
\node[ext](v2) at (2,0) {2};
\node[int] (w1) at (0,.7) {};
\node[int] (w2) at (0.5,1.4) {};
\node[int] (w3) at (1,.7) {};
\node[int] (w5) at (2,.7) {};
\draw[-latex] (v1) edge (w1) edge (w3) edge (w5) 
         (v2) edge (w5) edge (w1) edge (w3)
         (w1) edge (w2) 
         (w3) edge (w2);
\end{tikzpicture}
\]
Furthermore, for the purpose of computing the cup product on the hairy graph homology, we may disregard exact such forest, i.e., we may understand $m_2$ as an element of the quotient
\[
m_2 \in \dGraphsor_2(2) / im(\delta).
\]
However, then the bivalent trees may be interpreted as Lie trees, i.e., one considers them modulo the Jacobi relation.
Furthermore, $m_2$ must have the form 
\[
m_2 = 
\begin{tikzpicture}[baseline=-.65ex]
\node[ext](v1) at (0,0) {1};
\node[ext](v2) at (1,0) {2};
\end{tikzpicture}
+
\begin{tikzpicture}[baseline=-.65ex]
\node[ext](v1) at (0,0) {1};
\node[ext](v2) at (1,0) {2};
\node[int](w) at (.5,0.7) {};
\draw[-latex] (v1) edge (w) (v2) edge (w);
\end{tikzpicture}
+ (\text{terms with more vertices})
\]
and on top of this $m_2$ must be associative:
\[
m_2(m_2(-,-),-) = m_2(-,m_2(-,-)).
\]
However, it is well known that these conditions already uniquely determine $m_2$.
Concretely, $m_2$ may be identified with the Poincar\'e-Birkhoff-Witt (star) product. 
Let us briefly describe this product, see \cite{Kathotia} for a more detailed exposition. Suppose $\alg g$ is any Lie algebra. Then by the Poincar\'e-Birkhoff-Witt Theorem we have an isomorphism 
\[
\Phi : S\alg g\to \mU \alg g
\]
by symmetrization. Pulling back the (in general) non-commutative product of $\mU \alg g$ along $\Phi$ yields a (generally) non-commutative product $\star$ on $S\alg g$, the Poincar\'e-Birkhoff-Witt product.
This product may be given formally as a bi-differential operator
\[
f\star g = \sum_T c_T D_T(f,g)
\]
where the sum is over Lie forests with leaves decorated by 1 or 2, $D_T$ is a natural bidifferential operator associated to the forest, and $c_T$ are combinatorial constants, for which one does not have a nice closed form expression.
Our desired element $m_2 \in \dGraphsor_2(2) / im(\delta)$ then has the form
\[
\sum_T c_T T
\]
(up to exact elements).

\subsection{Example: Cup product with one-hair graphs}
There is one piece of the PBW product for which there is a nice closed formula. The weights of the part of degree one in one argument, i.e., the number $c_{T_n}$ for trees of the form 
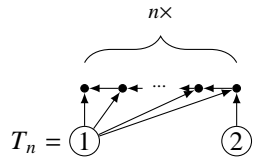
\[
T_n = 
\begin{tikzpicture}[baseline=-.65ex]
\node[ext](v1) at (0,0) {1};
\node[ext](v2) at (2,0) {2};
\node[int] (w1) at (0,.7) {};
\node[int] (w2) at (0.5,.7) {};
\node (w3) at (1,.7) {$\scriptstyle \cdots$};
\node[int] (w4) at (1.5,.7) {};
\node[int] (w5) at (2,.7) {};
\draw[-latex] (v1) edge (w1) edge (w2) edge (w4) edge (w5) 
         (v2) edge (w5)
         (w2) edge (w1) 
         (w3) edge (w2)
         (w4) edge (w3)
         (w5) edge (w4);
\draw [decorate,decoration={brace,amplitude=10pt,raise=4pt},yshift=0pt]
(0,.9) -- (2,.9) node [black,midway,yshift=0.8cm] {\footnotesize $n\times$ };
\end{tikzpicture}
\]
is known to be \cite[eqn. (2.21)]{Kathotia}
\[
c_{T_n} = \frac{B_n}{n!}
\]
where $B_n$ is the $n$-th Bernoulli number.
We can use this to compute the cup product of any hairy graph $\Gamma$ with a hairy graph $\Gamma_1$ with only one hair.
We obtain that this cup product is 
\[
\Gamma\cup \Gamma_1 
=
\begin{tikzpicture}[baseline=-.65ex]
\node (v) at (0,0) {$\Gamma$};
\node at (0,-.3){$\scriptstyle \cdots$};
\draw (v) edge +(-.5,-.5) edge +(-.4,-.5) edge +(.4,-.5) edge +(.5,-.5);
\end{tikzpicture}
\begin{tikzpicture}[baseline=-.65ex]
\node (v) at (0,0) {$\Gamma_1$};
\draw (v) edge +(0,-.5) ;
\end{tikzpicture}
+
\sum_{n=1}^\infty
\frac{B_n}{n!}
\begin{tikzpicture}[baseline=-.65ex]
\node[int] (v1) at (0,0) {};
\node[int] (v2) at (0.5,0) {};
\node (v3) at (1,0) {$\scriptstyle \cdots$};
\node[int] (v4) at (1.5,0) {};
\node (v5) at (2.5,0) {$\Gamma_1$};
\node (w) at (1,1) {$\Gamma$};
\draw (v1) edge (v2) edge +(-.5,-.5)
          (v3) edge (v2) edge (v4) 
          (v5) edge (v4)
          (w) edge (v1) edge (v2) edge (v4) edge +(-.5,0) edge +(-.5,.25) edge +(-.5,-.25);
\draw [decorate,decoration={brace,amplitude=10pt,mirror,raise=4pt},yshift=0pt]
(0,-.3) -- (1.5,-.3) node [black,midway,yshift=-0.8cm] {\footnotesize $n\times$ };
\end{tikzpicture}\, .
\]

\end{document}